\documentclass[12pt,reqno]{amsart}
\usepackage{amssymb}
\usepackage[T1]{fontenc}
\usepackage[latin1]{inputenc} 
\usepackage{dsfont, mathrsfs}
\usepackage{a4wide}  
\usepackage{stmaryrd}
\usepackage{amsthm}
\usepackage{xspace}
\usepackage[hyphenbreaks]{breakurl}
\usepackage[hyphens]{url}

\usepackage[colorlinks, pagebackref, hyperindex]{hyperref}

\hypersetup{
    colorlinks=true,         
    linkcolor=blue,          
    citecolor=green,         
    filecolor=magenta,       
    urlcolor=cyan            
}

\theoremstyle{plain}
\newtheorem{theorem}{Theorem}[section]
\newtheorem{lemma}[theorem]{Lemma}

\theoremstyle{definition}
\newtheorem{definition}[theorem]{Definition}
\newtheorem{example}[theorem]{Example}

\newtheorem{question}[theorem]{Question}

\theoremstyle{remark}
\newtheorem{remark}[theorem]{Remark}

\newcommand{\Cc}{\mathbb{C}}

\newcommand{\Ee}{\mathbb{E}}

\newcommand{\Nn}{\mathbb{N}}
\newcommand{\Pp}{\mathbb{P}}

\newcommand{\Rr}{\mathbb{R}}

\newcommand{\Uu}{\mathbb{U}}

\newcommand{\Yy}{\mathbb{Y}}
\newcommand{\Zz}{\mathbb{Z}}

\newcommand{\Un}{\mathds{1}}

\newcommand{\Be}{\mathcal{B}}
\newcommand{\Ce}{\mathcal{C}}

\newcommand{\He}{\mathcal{H}}

\newcommand{\Je}{\mathcal{J}}

\newcommand{\Le}{\mathcal{L}}

\newcommand{\Pe}{\mathcal{P}}

\newcommand{\Ze}{\mathcal{Z}}

\newcommand{\Ns}{\mathscr{N}}

\newcommand{\Us}{\mathscr{U}}

\newcommand{\Mg}{\mathfrak{M}}

\newcommand{\Rg}{\mathfrak{R}}
\newcommand{\Sg}{\mathfrak{S}}

\newcommand{\eg}{\mathfrak{e}}

\newcommand{\chib}{{\boldsymbol{\chi}}}

\newcommand{\ab}{{\boldsymbol{a}}}
\newcommand{\bb}{{\boldsymbol{b}}}

\newcommand{\xb}{{\boldsymbol{x}}}
\newcommand{\yb}{{\boldsymbol{y}}}
\newcommand{\zb}{{\boldsymbol{z}}}

\newcommand{\ensemble}[1]{ \left\lbrace #1 \right\rbrace } 
\newcommand{\prth}[1]{\!\left( #1 \right) }
\newcommand{\crochet}[1]{\!\left[ #1 \right] }  
\newcommand{\intcrochet}[1]{\llbracket #1 \rrbracket} 
\newcommand{\abs}[1]{\left| #1 \right|}  
\newcommand{\norm}[1]{\left| \! \left| #1 \right| \! \right|}

\newcommand{\Esp}[1]{ \Ee \prth{ #1 } }  
\newcommand{\Prob}[1]{ \Pp \prth{ #1 } }

\newcommand{\Espr}[2]{ \Ee_{#1}\!\prth{ #2 } }  
\newcommand{\Proba}[2]{ \Pp_{#1}\prth{ #2 } }

\def\divise{  {\circ\!\vert} }
\def\multiplede{  {\vert\!\circ} } 
\def\inv{^{-1}}

\def\longlongrightarrow{\hspace{+0.1ex} - \hspace{-1.1ex} - \hspace{-1.1ex} - \hspace{-1.1ex}\longrightarrow  } 

\newcommand{\tendvers}[2]{ \underset{#1 \rightarrow #2}{\longlongrightarrow} }  
\newcommand{\cvlaw}[2]{ \stackrel{\Le}{\underset{#1 \, \rightarrow \, #2}{\longlongrightarrow}} }

\newcommand{\bracket}[1]{\langle #1 \rangle}
\newcommand{\Unens}[1]{ \Un_{ \!\!\; \ensemble{#1} } }
\def\eqlaw{\stackrel{\Le}{=}}

\def\Pl{\mathcal{P}\ell}
\def\Ch{\operatorname{Ch}}
\def\ct{\operatorname{ct}}
\def\Mod{\operatorname{mod}}
\def\Li{\operatorname{Li}}

\def\BJ{\Be\Je}

\def\geq{\geqslant}
\def\leq{\leqslant}

\def\Re{\Rg \eg}

\let\oldforall\forall
\def\forall{\oldforall\,} 

\let\oldexists\exists
\def\exists{\oldexists\,}

\newcommand{\emailhref}[1]{ \email{\href{mailto:#1}{#1}} }

\begin{document}


\title[Random Dirichlet characters]{ Random characters under the $L$-measure, I : \\ Dirichlet characters} 


\author{Yacine Barhoumi-Andréani}
\address{Department of Statistics, University of Warwick, Coventry CV4 7AL, U.K.}
\emailhref{y.barhoumi-andreani@warwick.ac.uk}


\date{\today}


\begin{abstract}
We define the $L$-measure on the set of Dirichlet characters as an analogue of the Plancherel measure, once considered as a measure on the irreducible characters of the symmetric group. 

We compare the two measures and study the limit in distribution of characters evaluations when the size of the underlying group grows. These evaluations are proven to converge in law to imaginary exponentials of a Cauchy distribution in the same way as the rescaled windings of the complex Brownian motion. This contrasts with the case of the symmetric group where the renormalised characters converge in law to Gaussians after rescaling (Kerov Central Limit Theorem).
\end{abstract}

\maketitle

\vspace{-0.4cm}

\tableofcontents

\section{Introduction}

\subsection{Asymptotic representation theory of $ (\Zz/q\Zz)^\times $ }
For $ q \in \Nn^* $, set
\begin{align}\label{Def:Z/qZstar}
G_q := (\Zz/q\Zz)^\times
\end{align}

The goal of this article is to study the properties of characters of $ G_q $ when $q \to +\infty $ and when the characters are selected at random according to the \textit{$L$-measure} that we introduce in definition \ref{Definition:L-measure}. 

This article focuses on the case of \textit{Dirichlet characters} modulo $q$ that are multiplicative group morphisms $ \chi : G_q \to \Cc^\times $ extended to $ \Zz/q\Zz $ by setting $ \chi(n) = 0 $ if $ n \in (\Zz/q\Zz) \setminus (\Zz/q\Zz)^\times $ and finally extended by periodicity to $ \Zz $ by setting $ \chi(n) := \chi(n \!\!\mod q ) $ (see e.g. \cite[§ 5]{IwanecKowalski}). These maps were used by Dirichlet to prove his celebrated theorem on the infinitude of primes in arithmetic progressions.

Dirichlet characters have the following properties :
\begin{enumerate}
\item $ \chi $ is periodic modulo $q$ : $ \chi(n + q) = \chi(n) $ for all $ m, n \in \Nn $, 
\item $ \chi $ is completely multiplicative : $ \chi(mn) = \chi(m) \chi(n) $ for all $ m, n \in \Nn $, 
\item $ \chi(n) \neq 0 $ if and only if $ \gcd(n, q) = 1 $.
\end{enumerate}


Note that these properties imply that $ \chi(1) = 1 $, as $ \chi(1) = \chi(1\times 1) = \chi(1)^2 $ and $ \chi(1) \neq 0 $ since $ \gcd(1, q) = 1 $.

We define $ \widehat{G}_q $ to be the set of Dirichlet characters modulo $ q $. Apart from the value $0$, these characters take values in the $ \varphi(q) $-roots of unity $ \ensemble{ e^{2i \pi k/\varphi(q) }, k \geq 0 } $, where $ \varphi $ is Euler's totient function. There are exactly $ \varphi(q) $ such characters.

The $ L $-function attached to a Dirichlet character $\chi$ is the following function defined for all $ s \in \ensemble{ \Re > 1 } $ by (see e.g. \cite{Tenenbaum})
\begin{align}\label{Def:L-Function}
L_s(\chi) := \sum_{n \geq 1 } \frac{\chi(n) }{ n^s }
\end{align}

We consider these functions as linear forms in the character $ \chi $, hence the choice of notation compared with the usual one that writes $ L(s, \chi) $.

\begin{definition}\label{Definition:L-measure} 

We define the \textit{$L$-measure} on $ \widehat{G}_q $ by
\begin{eqnarray}\label{Def:DirichletMeasure}
\Proba{\! s, q}{ \chi } := \frac{ \abs{L_s(\chi)}^2 }{ Z_{s, q} }, \quad\quad  Z_{s, q} := \sum_{ \eta \in \widehat{G}_q } \abs{L_s(\eta)}^2, \qquad s \in \ensemble{ \Re > 1 }
\end{eqnarray}
%
%
%
%
\end{definition}

This measure can be written in the following way using the infamous Euler formula~\eqref{Eq:EulerProdDirichletSum}~:
\begin{eqnarray}\label{Def:DirichletMeasureHamiltonian}
\Proba{\!s, q}{ \chi } := \frac{ 1 }{ Z_{s, q} } \exp\prth{ - 2 \He_s(\chi) }, \quad\quad  \He_s(\chi) := \sum_{p \in \Pe } \log\abs{1 - p^{-s } \chi(p) }
\end{eqnarray}
where $ \Pe $ is the set of prime numbers. 

We can thus interpret \eqref{Def:DirichletMeasure} in the setting of statistical mechanics as a system of particles on a circle (whose positions are given by the angles of the character evaluated in prime numbers) submitted to a particular logarithmic confinement potential at inverse temperature~$2$.

Let $ \chib \equiv \chib^{(q)} $ denote the canonical evaluation on $ \widehat{G}_q $ in the Dynkin formalism, namely
\begin{align}\label{Def:CanonicalEvaluation}
\chib_k (\eta) = \eta(k), \quad\quad\quad \forall \eta \in \widehat{G}_q
\end{align}

We will be interested in the behaviour of the random variables $ \chib_k $ for a fixed integer $k$. 
The main theorem of this paper states


\begin{theorem}[Convergence in law of the evaluations]\label{Theorem:ConvergenceAngles}
Let $ k \in \intcrochet{2, q} $ be a fixed integer and $ s \in (1, +\infty) $. Then, under $ \Pp_{s, q} $, the following convergence in law is satisfied
\begin{align*}
\chib_k \cvlaw{q}{ +\infty } e^{ i s \log(k) \, \Ce }
\end{align*}
where $ \Ce $ is a standard Cauchy-distributed random variable of density $ x \mapsto \frac{1}{\pi} \frac{1}{1 + x^2 }  $.
\end{theorem}


This theorem is proven in section \ref{SubSec:FluctuationsEvaluations}. An immediate striking comparison with the winding number of the complex Brownian motion can be made~:~write $ Z_t := 1 + W_t + i W'_t = R_t e^{i \Theta_t } $ with $ (W, W') $ two independent real Brownian motions. Here, $ (\Theta_t)_t $ is a continuous determination of the argument of $ Z $ around $0$. Then, one has (see \cite{Spitzer} or \cite[X-4.1]{RevuzYor}, \cite[ch. 7]{MansuyYor})
\begin{align}\label{Eq:WindingNumberCvLaw}
\frac{ \Theta_t }{ \log(\sqrt{t}) }  \cvlaw{t}{ +\infty } \Ce 
\end{align}

Spitzer proved this last convergence using an explicit computation of the Fourier transform of $ \Theta_t $ (see \cite[X-4.1]{RevuzYor} for another proof). We will prove this convergence in the same vein, with a direct computation of the Laplace transform of $ \chib_k $. 

The similarity stops nevertheless here : the windings of $ (Z_t)_t $ around several points $ z_1, \dots, z_k $ of the complex plane do not converge in law after rescaling to independent Cauchy random variables. A precise description of the limiting joint distribution is given in \cite[ch. XIII cor. 3.9]{RevuzYor} and involves a mixing by a random variable that creates a dependency in the limiting angles. In the case of evaluations of random Dirichlet characters, one has the following

\begin{theorem}[Fluctuations of joint evaluations]\label{Theorem:JointEvaluationFluctuations}

Let $ s \in (1, +\infty ) $. Let $ \ell \geq 1 $ and $ (p_j)_{1 \leq j \leq \ell} $ be fixed prime numbers. Then, under $ \Pp_{s, q} $, we have the following convergence in distribution
\begin{align*}
( \chib_{ p_1}, \dots, \chib_{ p_\ell } ) \cvlaw{q}{ +\infty } \prth{ e^{i s \log(p_1) \, \Ce_1} , \dots, e^{i s  \log(p_\ell) \, \Ce_\ell } }
\end{align*}
where $ (\Ce_1, \dots, \Ce_\ell) $ are independent Cauchy-distributed random variables.
\end{theorem}

It is enough to consider the evaluations in prime numbers $ (\chib_p)_{p \in \Pe} $ by multiplicativity of characters. More general evaluations will be dependent, provided that the numbers in which they are evaluated are not coprime.

\subsection{Motivations}


Let $ n \in \Nn $. The Plancherel measure on the set $ \Yy_n $ of Young diagrams with $n$ boxes is defined by (see e.g. \cite{Kerov, OlshanskiSurvey})
\begin{align*}
\Pl_n(\lambda) := \frac{ d_\lambda^2 }{n!} 
\end{align*}

Here, $ d_\lambda $ denotes the dimension of the irreducible $ \Sg_n $-module indexed by the Young diagram $ \lambda \in \Yy_n $ and $ \Sg_n $ is the symmetric group ; this quantity has an explicit expression in terms of $ \lambda $, but it can also be written in terms of characters of $ \Sg_n $ as its irreducible characters are in bijection with Young diagrams (see \cite[I-7]{MacDo}). For $ \lambda \in \Yy_n $, let $ \chi^\lambda : \Sg_n \to \Zz $ be the associated irreducible character. A classical formula states (see \cite[ch. I-7 (7.6), ch. I-6 ex. 2 (a)]{MacDo})
\begin{align*}
d_\lambda = \chi^\lambda(id_n)
\end{align*}
where $ id_n \in \Sg_n $ is the identity permutation. This last formula allows to write the Plancherel measure as a measure on characters by setting
\begin{align}\label{Def:PlancherelMeasureOnCharacters}
\Pl_n(\chi) := \frac{ \chi(id_n)^2 }{n!}
\end{align}

Character evaluations form a natural set of observables for this measure. Using the Dynkin formalism, we define the canonical evaluation $ \chib $ on $ \widehat{\Sg}_n $ by
\begin{align*}
\chib_\sigma(\eta) := \eta(\sigma), \quad\quad\quad \forall \eta \in \widehat{\Sg}_n, \quad \forall \sigma \in \Sg_n
\end{align*}

A natural question concerns the behaviour of $ \chib_\sigma $ under $ \Pl_n $ for a certain $ \sigma \in \Sg_n $ and $ n \to +\infty $ ; such types of questions gave birth to the field of \textit{asymptotic representation theory} \cite{Kerov}. In the case where $ \sigma $ is a $k$-cycle $ c_k $ for a fixed integer $k$, one has the fondamental result due to Kerov \cite{IvanovOlshanskiKerov} independently discovered by Hora \cite{HoraKerov} :

\begin{theorem}[Gaussian fluctuations of character evaluations (Kerov CLT)]\label{Theorem:KerovCLT} Let $ k \geq 2 $ be a fixed integer and $ c_k $ a $k$-cycle. Then, under $ \Pl_n $, one has the following convergence in law
\begin{align*}
\frac{1}{n^{k/2}  }   \frac{ \chib_{ c_k } }{ \chib_{ id_n } }     \cvlaw{n}{+\infty } \sqrt{k } G_k \qquad G_k \sim \Ns(0, 1)
\end{align*}

Moreover, different evaluations $ (n^{-k/2} \chib_{ c_k } /\chib_{ id_n })_{2 \leq k \leq r} $ for a fixed integer $r$ converge to independent Gaussian random variables.
\end{theorem}

This result concerns the Gaussian fluctuations of the \textit{renormalised} character, namely $ \chi / \chi(id_n) $. It has been extended in various ways, by computing the speed of convergence in the total variation distance (see \cite{Fulman:SteinPlancherel}), by letting $k$ depend on $n$ in the cycle $ c_k $ (for $  k = O(\sqrt{n}) $, see \cite{Sniady:Gaussian} and references cited), by changing the cycle $c_k$ with a product of cycles, in which case the limiting distribution changes radically (see \cite{HoraKerov, IvanovOlshanskiKerov}) or by changing the measure (see e.g. \cite{FerayMeliot, Kerov-q-Plancherel, BorodinOlshanskiZmeas, AdinPostnikovRoichman, MeliotGelfand, Biane:ApproximateFactorisation} and references cited).

This article is in the lineage of these results ; it aims at giving another direction of generalisation of this last analysis by replacing the group $ \Sg_n $ with the group $ (\Zz/q \Zz)^\times $. The analogue of the Plancherel measure on $ (\Zz/q \Zz)^\times $ will be the uniform measure, which is a limiting case of $ L $-measure that we will study separately in section \ref{SubSec:OtherMeasure}. The case of the $ L $-measure corresponds to the aforementioned other types of measures, particular cases of the \textit{Schur measure} that we describe in the next section. In this framework, theorems \ref{Theorem:ConvergenceAngles} and \ref{Theorem:JointEvaluationFluctuations} are clear analogues of the Kerov CLT, the equivalence betwen disjoint cycles and prime numbers being made in accordance to the usual analogies between permutations and primes (see e.g. \cite{Grandville} or \cite[p. 22]{ArratiaBarbourTavare}). These theorems have nevertheless some disparities~:~random Dirichlet characters need not be renormalised, and the limiting random variables involve the Cauchy distribution (the stable law of index $1$) and not the Gaussian one (the stable law of index $2$).

\subsection{Natural analogues of the Plancherel measure on other groups}

The first point to investigate when trying to generalise theorem \ref{Theorem:KerovCLT} to other groups concerns the natural analogue of the Plancherel measure \eqref{Def:PlancherelMeasureOnCharacters}. Given its form, the Plancherel measure on $ \widehat{G}_q $ is
\begin{align}\label{Def:UniformMeasureDirChar}
\Pp(\chi) := \frac{  \chi(1)^2 }{ \sum_{ \eta \in \widehat{G}_q } \chi(1)^2 } = \frac{1}{\varphi(q) }
\end{align}
as $ \chi(1) = 1 $. This is thus the uniform measure on $ \widehat{G}_q $. This measure corresponds to a limiting type of $L$-measure, where we have set $ s \to +\infty $ in \eqref{Def:DirichletMeasure} (see section \ref{SubSec:OtherMeasure}). For this degenerate type of $ L $-measure, the limiting behaviour of the characters evaluations is given by the following theorem proven in section \ref{SubSec:OtherMeasure} :

\begin{theorem}[Limit in law of the evaluations under $ \Pp_{\infty, q} $]\label{Theorem:LimitEvaluationsUniformMeasure} The following convergence in law is satisfied for all fixed integer $ k \in \intcrochet{2, q} $ 
\begin{align*}
\chib_k   \cvlaw{q}{ + \infty } e^{i 2\pi  U}, \qquad U \sim \Us\!\prth{\, \crochet{0, 1} }
\end{align*}

Moreover, $ (\chib_p)_{p \in \Pe} $ converges in law to a vector of independent random variables. 
\end{theorem}

$ $

More generally, the Plancherel measure and a wide class of measures studied in \cite{FerayMeliot, Kerov-q-Plancherel, BorodinOlshanskiZmeas, AdinPostnikovRoichman, MeliotGelfand, Biane:ApproximateFactorisation} are particular specialisations of the following Schur measure restricted to $ \widehat{\Sg}_n $ :
\begin{align}\label{Def:RestrictedSchurMeasure}
\Mg^{(n)}_{X, Y}(\lambda) := \frac{ s_\lambda(X) s_\lambda(Y) }{ h_n\crochet{XY} } \Unens{ \lambda \in \Yy_n}, \qquad h_n\crochet{XY} := \sum_{ \lambda \in \Yy_n } s_\lambda(X) s_\lambda(Y) 
\end{align}

Here, $ s_\lambda(X) $ is the Schur function \cite[ch. 1-3]{MacDo} evaluated in an alphabet $ X $ and the normalisation constant $ h_n\crochet{XY} $ is a specialisation of the complete homogeneous functions \cite[ch. 1-2]{MacDo}. The prominent Schur measure on the set of all Young diagrams introduced by Okounkov in \cite{OkounkovSchurMes} is an independent randomisation\footnote{By a random variable $ B_{X, Y} $ satisfying $ \Prob{ B_{X, Y} = n } =   h_n\crochet{XY}/H\crochet{XY}  $ with $ H\crochet{XY} := \sum_{n \geq 0} h_n\crochet{XY} $~. The Schur measure writes $ \Mg_{X, Y}(\lambda) := \sum_{n \geq 0 }\Prob{ B_{X, Y} = n } \Mg^{(n)}_{X, Y}(\lambda)  = s_\lambda(X) s_\lambda(Y) /  H\crochet{XY} $.} of \eqref{Def:RestrictedSchurMeasure}.

The Schur measure writes as a measure on $ \widehat{\Sg}_n $ using the Frobenius characteristic (see \cite[I-7]{MacDo})
\begin{align*}
s_\lambda(X) = \Ch_X(\chi^\lambda) := \frac{1}{n!} \sum_{\sigma \in \Sg_n } \chi^\lambda(\sigma) \, p_{ \ct(\sigma) }(X)  
\end{align*}
where $ p_{\ct(\sigma) }(X) $ are the \textit{power functions}, a particular type of symmetric functions (see \cite[ch. 1-2]{MacDo}) indexed by the cycle-type of a permutation. Write $ \sigma = \sigma_1 \dots \sigma_{C(\sigma)} $ where $ (\sigma_k)_{1 \leq k \leq C(\sigma)} $ are the disjoint cycles of $ \sigma $ and $ C(\sigma) $ is the total number of cycles. Then, the power functions can be expressed in terms of a family $ (p_\ell(X))_{\ell \geq 1} $ associated with the lenght of a cycle
\begin{align*}
p_{ \ct(\sigma) }(X) = \prod_{k = 1}^{ C(\sigma) } p_{ \abs{ \sigma_k } }(X) = \prod_{\ell \geq 1} p_\ell(X)^{ m_\ell(\sigma) }
\end{align*}
with $ m_\ell(\sigma) := \sum_{k \geq 1 } \Unens{  \abs{ \sigma_k } = \ell } $ and $ \abs{\sigma_k} $ is the size of the cyle $ \sigma_k $ (the number of its elements).

We can thus write 
\begin{align*}
\Ch_X(\chi ) = \frac{1}{n!} \sum_{\sigma \in \Sg_n } \chi (\sigma) \, \prod_{k \geq 1} p_\ell(X)^{ m_\ell(\sigma) }  
\end{align*}
in the same way as 
\begin{align*}
L_s(\chi ) = \sum_{ n \in \Nn^* } \chi (n) \, \prod_{p \in \Pe } p^{- s v_p(n) }  
\end{align*}

The point of importance is to notice that one can write the measure \eqref{Def:RestrictedSchurMeasure} as a measure on $ \widehat{\Sg}_n $ by writing, in the same vein as for the Plancherel measure
\begin{align*}
\Mg^{(n)}_{X, Y}(\chi) = \frac{ \Ch_X(\chi ) \Ch_Y(\chi ) }{ h_n\crochet{XY} }
\end{align*}

The Frobenius characteristic $ \Ch_X $ is a linear form on $ \widehat{\Sg}_n $ and so is the evaluation $ \chib_{id_n} $~; the common points in all the precedent measures is thus (1) their writing as a product of two linear forms and (2) the coefficients of each linear form that write as a product over the considered structures (cycles or primes).

If one replaces the characters of $ \Sg_n $ by characters of another finite group $ G $ whose elements have a natural notion of decomposition into elementary structures (primes, cycles, etc.), a natural analogue of these measures can hence be defined by means of a product of two real linear forms on $ \widehat{G} $ of the form
\begin{align}\label{Def:L-SchurMeasure}
\begin{aligned} 
& \Pp_{L_1, L_2}(\chi) := \frac{ L_1(\chi ) L_2(\chi ) }{ \Ze(\widehat{G}) }, \quad  \Ze(\widehat{G}) := \sum_{\chi \in \widehat{G} } L_1(\chi ) L_2(\chi ) \\
&  L_i(\chi) \equiv L_\zb (\chi) := \sum_{g \in G} \chi(g) \prod_{a \in \Ce(g) } z_a^{ m(a) }, \quad i \in \ensemble{1, 2}, \quad \zb := \prth{ z_a }_a
\end{aligned}
\end{align}
where $ \Ce(g) $ is a particular set characteristic of a certain decomposition of $ g \in G $ and the $ z_a $'s are real numbers.

The measure \eqref{Def:DirichletMeasure} is clearly a particular case of \eqref{Def:L-SchurMeasure}, but it has only a one-dimensional free parameter $s$ ; from this perspective, it is more an analogue of the $z$-measure \cite{BorodinOlshanskiZmeas} or the $q$-Plancherel measure \cite{FerayMeliot, Kerov-q-Plancherel} than the Schur measure itself. We will come back on this last point in section \ref{Section:LMesureGenerale}.

\subsection{Plan}

The plan of this article is the following : after some reminders on characters of both $ \Sg_n $ and $ G_q $, we start by comparing random characters under the aforementioned measures and show that they share a similarity of structure ; their evaluations can be written by operators of dilation in both cases, acting on different spaces of square integrable functions. We then study the asymptotic properties of $ \chib_k $ 
under $ \Pp_{s, q} $ and prove theorems \ref{Theorem:ConvergenceAngles} and \ref{Theorem:JointEvaluationFluctuations}. We moreover give the speed of this convergence in the Wasserstein distance (lemma \ref{Lemma:SpeedOfConvergence}). We treat in section \ref{SubSec:OtherMeasure} the case of the uniform measure \eqref{Def:UniformMeasureDirChar} that corresponds to the limiting case of the $L$-measure $ \Pp_{\! \infty, q} $. Last, we treat in section \ref{Section:LMesureGenerale} the case of the analogue of the Schur measure on $ \widehat{G}_q $. We conclude with several questions of interest, extensions and future developments. 

\vspace{+0.2cm}

This article is the first of a series that investigates $L$-measures on sets of more general characters such as Hecke characters and automorphic representations of $ GL(n) $.

\vspace{+0.22cm}

\subsection{Notations}
We write $ a \divise b $ for ``$a$ divides $b$'' and $ b  \multiplede a $ for ``$b$ multiple of $a$''. We will make a constant use of the multi-index notation $ x^\lambda := \prod_{k \geq 1} x_k^{ \lambda_k } $ if $ x = (x_1, x_2, \dots) $ and $ \lambda = (\lambda_1, \lambda_2, \dots) $. We also set $ \lambda ! := \prod_{k \geq 1} \lambda_k ! $ for the multi-index factorial. Last, we set $ \abs{\lambda} = \sum_{k \geq 1} \lambda_k $ for a multi-index. The raising factorial is defined as $ x^{ \uparrow n } := x(x + 1)\cdots (x + n - 1) $ for $ x \in \Cc $ and $ n \geq 1 $.

We will note $ \lambda \vdash n $ for $ \lambda \in \Yy_n $. The Vandermonde determinant is denoted by $ a_\delta(x) := a_\delta(x_1, \dots, x_n) $ if $ \delta \equiv \delta_n = (n - 1, n - 2, \dots, 1, 0) $ and defined by 
\begin{align*}
a_\delta(x) := \prod_{1 \leq i < j \leq n } (x_i - x_j)
\end{align*}

For a Laurent series $ f $ in $n$ variables, we write $ \crochet{ x^\lambda }(f) $ the coefficient of $ x^\lambda $, i.e. $ \crochet{ x^\lambda }(f) = \int_{ \Uu^n } x^\lambda \overline{ f(x) } dm(x) $ with $ dm(x) := \prod_{k \geq 1} \frac{dx_k}{ 2i \pi  x_k } $ and $ \Uu $ the unit circle if this last integral converges.

The Euler's totient function is denoted by $ \varphi $ and defined for all $ n \geq 1 $ by
\begin{align*}
\varphi(n) := n \prod_{ p \divise n } \prth{1 - \frac{1}{p} }
\end{align*}

Here, the product is on the set of prime numbers $ \Pe $. Every integer $ n \in \Nn^* $ decomposes into a product of primes according to
\begin{align*}
n = \prod_{p \in \Pe } p^{ v_p(n) }
\end{align*}
where $ v_p(n) $ is the $p$-adic valuation of $n$. For a sequence of complex numbers $ (a_p)_{p \in \Pe} $ such that $ \abs{a_p } \leq 1 $, one has the formula, for all $ s \in \ensemble{\Re > 1} $
\begin{align}\label{Eq:EulerProd=SumGeneral}
\prod_{p \in \Pe } \prth{1 - p^{-s} a_p }\inv = \sum_{n \geq 1} \frac{1}{n^s} \prod_{p \in \Pe } a_p^{ v_p(n) }
\end{align}

In particular, for a character $ \chi \in \widehat{G}_q $, using the prime factor decomposition and the morphism property of $ \chi $, we get the well known Euler formula
\begin{align}\label{Eq:EulerProdDirichletSum}
\prod_{p \in \Pe } \prth{1 - p^{-s} \chi(p) }\inv = \sum_{n \geq 1} \frac{1}{n^s} \prod_{p \in \Pe } \chi(p)^{ v_p(n) } = \sum_{n \geq 1} \frac{1}{n^s} \chi(n)
\end{align}

Note that as $ \chi(p) = 0 $ if $ p \divise q $, the product on the LHS of the last formula is in fact on the set of primes that do not divide $q$.

Last, we set $ \Un(n) := \Unens{n \geq 1} $ and we define the multiplicative convolution $*$ of two sequences $ (a_n)_{n \geq 1} $ and $ (b_n)_{n \geq 1} $ by
\begin{align*}
a * b(n) := \sum_{k, \ell \geq 1} a_k b_\ell \Unens{k\ell = n } = \sum_{ \ell \divise n } a_\ell b_{n / \ell} = \sum_{ \ell \divise n } a_{n / \ell} b_\ell
\end{align*}

In particular, we note $ a^{* k}(n) := a*\cdots * a(n) $ ($k$ times). We moreover recall that Dirichlet series have the following morphism property for the multiplicative convolution, for $ z \in \ensemble{\Re > 1} $ and provided that the following series converge 
\begin{align}\label{Eq:ConvolutionMorphism}
\sum_{n \geq 1} \frac{a * b(n) }{ n^z } = \sum_{n \geq 1} \frac{a (n) }{ n^z } \sum_{m \geq 1} \frac{ b(m) }{ m^z }
\end{align}

\section{From $ \widehat{\Sg}_n $ to $ \widehat{G}_q $} 

\subsection{Reminders on the irreducible characters of $ \Sg_n $}

Let $ \chi^\lambda \in \widehat{\Sg}_n $ be the irreducible character associated to $ \lambda \vdash n $. For $ \sigma \in \Sg_n $, the value of $ \chi^\lambda(\sigma) $ only depends on the cycle type $ \mu := \ct(\sigma) \vdash n $. We denote this value by $ \chi^\lambda_\mu $. The canonical evaluation on the probability space $ (\Yy_n, \Pl_n) $ thus becomes $ \chi_\mu : \lambda \mapsto \chi^\lambda_\mu $ . A classical formula for the evaluation of $ \chi^\lambda_\mu $ is given by \cite[ch. I-7 (7.7)]{MacDo} 
\begin{eqnarray}\label{Eq:CharacterWithSchur}
\chi^\lambda_\mu = \bracket{ s_\lambda, p_\mu }
\end{eqnarray}

Here, $ p_\mu $ is the power function, and the scalar product can be realised as the classical $ L^2 $ scalar product for the Haar measure of the unitary group. One has also the following

\begin{lemma}[\cite{MacDo}, (7.8) p. 114] Write $ \crochet{ x^\lambda }(f) $ for the coefficient of $ x^\lambda := \prod_{k \geq 1} x_k^{\lambda_k } $ in the Laurent series $ f $ in $n$ variables, i.e. $ \crochet{ x^\lambda }(f) = \int_{ \Uu^n } x^\lambda \overline{ f(x) } dm(x) $ with $ dm(x) := \prod_{k \geq 1} \frac{dx_k}{ 2i \pi  x_k } $ and $ \Uu $ the unit circle. Then
\begin{align*}
\chi^\lambda_\mu =  \crochet{ x^{\lambda + \delta} }\prth{ p_\mu a_\delta } = \int_{ \Uu^n } x^{\lambda + \delta} \overline{ p_\mu(x)  a_\delta (x) } dm(x)
\end{align*}
\end{lemma}


For the reader's benefit, we remind the proof of this result.

\begin{proof}

Using equation \eqref{Eq:CharacterWithSchur} and the fact that $ (s_\lambda)_\lambda $ is an orthonormal basis, one has
\begin{align*}
p_\mu = \sum_{\lambda \vdash n} \bracket{ s_\lambda, p_\mu } s_\lambda = \sum_{\lambda \vdash n} \chi^\lambda_\mu  s_\lambda
\end{align*}

One has moreover (see e.g. \cite{MacDo}) $ s_\lambda = a_{\lambda + \delta}/ a_\delta $, hence $ p_\mu a_\delta = \sum_{\lambda \vdash n} \chi^\lambda_\mu  a_{\lambda + \delta} $ and in particular, 
\begin{align*}
\crochet{ x^{\lambda + \delta} }\prth{ p_\mu a_\delta } = \sum_{\nu \vdash n} \chi^\nu_\mu  \crochet{ x^{\lambda + \delta} }\prth{ a_{\nu + \delta} }
\end{align*}

It thus remains to prove that $ \crochet{ x^{\lambda + \delta} }\prth{ a_{\nu + \delta} } = \Unens{ \lambda = \nu } $. But one has
\begin{align*}
\crochet{ x^{\lambda + \delta} }\prth{ a_{\nu + \delta} } & = \int_{ \Uu^n } x^{\lambda + \delta} \overline{ a_{\nu + \delta}(x) } dm(x) \\
			& = \sum_{\sigma \in \Sg_n } \varepsilon(\sigma) \int_{ \Uu^n } x^{\lambda + \delta} \overline{x^{\sigma(\nu + \delta)}} dm(x) \\
			& = \sum_{\sigma \in \Sg_n } \varepsilon(\sigma) \Unens{ \lambda + \delta = \sigma(\nu + \delta) } \\
			& = \Unens{ \lambda = \nu }
\end{align*}
since the only permutation allowing to have $ \lambda + \delta = \sigma(\nu + \delta) $ is the identity, the partitions $ \lambda + \delta $ and $ \nu + \delta $ having strictly increasing parts.
\end{proof}


\subsection{Moments of characters under the Plancherel measure}

\begin{lemma}[Structure of character evaluations] Let $ dm(z) := \prod_{ k \geq 1 } \frac{dz_k}{2 i \pi z_k} $. For $ f, g \in L^2(\Uu, m) \equiv L^2( m) $, define the scalar product
\begin{align*}
\bracket{f, g}_{L^2(\Uu^n) } := \int_{ \Uu^n } f(z) \overline{g(z) } dm(z) = \int_{ \crochet{0, 1}^n } f\prth{ e^{2 i \pi \theta_1 }, \dots, e^{2 i \pi \theta_n } } \overline{ g\prth{ e^{2 i \pi \theta_1 }, \dots, e^{2 i \pi \theta_n } } } \, d\theta_1 \dots d\theta_n
\end{align*}

For all $ \mu \vdash n $ and $ x \in \Uu $, set 
\begin{align*}
f_\mu(x) := p_\mu(x) a_\delta(x) 
\end{align*}

We define the operator of anti-dilatation by
\begin{align*}
\widetilde{\delta_x} g(u) := g( \overline{x_1} u_1, \overline{x_2} u_2, \dots, \overline{x_n} u_n )
\end{align*}
and the following operator  
\begin{align*}
\Le_\mu := \int_{ \Uu^n } \overline{ f_\mu (x) } \widetilde{\delta_x} \, dm(x) : g  \in L^2(m) \mapsto \int_{ \Uu^n }  \overline{ f_\mu (x) } g(\overline{x} \odot \cdot )  \, dm(x)   \in L^2(m)
\end{align*}

Using the fact that $ f_\mu \in L^2(m) $ (as $ f_\mu $ is a polynomial) and the Cauchy-Schwarz inequality, it is easily proven that $ \Le_\mu g \in L^2(m) $ if $  g \in L^2(m) $. 

Last, set
\begin{align*}
R_n(x) & := \sum_{\lambda \vdash n } \overline{x}^\lambda  \\
h_n(x) & := \Le_{1^n} R_n
\end{align*}

Then, we have
\begin{eqnarray}\label{Eq:EvaluationPlancherelWithOperator}
\Ee_{\Pl}^{(n)}\prth{ \chi_\mu^k \overline{\chi_\nu}^\ell  }  =  \frac{ \bracket{ \Le_\mu^k  h_n,  \Le_\nu^\ell h_n  }_{L^2(\Uu^n) }  }{ \bracket{  h_n, h_n  }_{L^2(\Uu^n) } } 
\end{eqnarray}

\end{lemma}


\begin{proof}
For $ k \geq 1 $ and $ \mu \vdash n $, one has 
\begin{align*}
\Ee_{\Pl}^{(n)}\prth{ \abs{ \chi_\mu }^{2 k} } & = \frac{1}{n!} \sum_{ \lambda \vdash n } d_\lambda^2 \prth{ \chi_\mu^\lambda }^{2 k} = \frac{1}{n!} \sum_{ \lambda \vdash n } \bracket{s_\lambda, p_{ 1^n  } }^2   \bracket{s_\lambda, p_\mu }^{2 k} \\
		  & = \frac{1}{n!} \sum_{ \lambda \vdash n } \int_{ \Uu^n } x^{\lambda + \delta} \overline{ p_{ 1^n  }(x)  a_\delta (x) } dm(x)  \overline{ \int_{ \Uu^n } y^{\lambda + \delta} \overline{ p_{ 1^n  }(y)  a_\delta (y) } dm(y) } \times \\
		  &    \hspace{+1cm} \prod_{\ell = 2}^{k + 1} \int_{ \Uu^n } x_\ell^{\lambda + \delta} \overline{ p_\mu(x_\ell)  a_\delta (x_\ell) } dm(x_\ell)   \overline{ \int_{ \Uu^n } y_\ell^{\lambda + \delta} \overline{ p_\mu(y_\ell)  a_\delta (y_\ell) } dm(y_\ell) } \\
		  & =  \frac{1}{n!} \int_{ (\Uu^n)^{2(k + 1)} } \prth{ \sum_{ \lambda \vdash n } \prod_{\ell = 1}^{k + 1} x_\ell^{\lambda + \delta} \overline{ y_\ell^{\lambda + \delta} } } \overline{ p_{ 1^n  }(x_1)  a_\delta (x_1) } p_{ 1^n  }(y_1)  a_\delta (y_1)   \times \\
		  &   \hspace{+1cm} \prod_{\ell = 2}^{k + 1} \overline{ p_\mu(x_\ell)  a_\delta (x_\ell) } p_\mu(y_\ell)  a_\delta (y_\ell)  \prod_{\ell = 1}^{k + 1} dm(x_\ell)dm(y_\ell)  
\end{align*}

Denote $ \xb := (x_1, \dots, x_{k + 1}) $ and $ \yb := (y_1, \dots, y_{k + 1}) $. Last, set
\begin{eqnarray}\label{Def:AutoreproducingKernelPermutations} 
K_n( \xb, \yb ) := \sum_{ \lambda \vdash n } \prod_{\ell \geq 1}  x_\ell^{\lambda + \delta} \overline{ \prod_{\ell \geq 1}  y_\ell^{\lambda + \delta} } 
\end{eqnarray}

In this last formula, the variables can have a different size. For instance, one can define
\begin{align*}
K_n( \xb, y ) := \sum_{ \lambda \vdash n } \prth{ \prod_{\ell \geq 1}  x_\ell^{\lambda + \delta} } \overline{  y ^{\lambda + \delta} }
\end{align*}

Then, one has
\begin{align*}
\Ee_{\Pl}^{(n)}\prth{ \abs{ \chi_\mu }^{2 k} } = \frac{1}{n!} \int_{ (\Uu^n)^{2(k + 1)} } K_n(\xb, \yb) \ \overline{ f_{1^n}\!\! \otimes \!\!  f_\mu^{\otimes k}(\xb) } \cdot  f_{1^n} \!\!\otimes\!\! f_\mu^{\otimes k}(\yb)  \ dm^{\otimes k + 1}(\xb) dm^{\otimes k + 1}(\yb)
\end{align*}

Using the formula
\begin{align*}
\int_{ \Uu^n  } z^{\lambda + \delta } \overline{ z^{ \nu + \delta } } \ dm (z) = \Unens{ \lambda = \nu  }
\end{align*}
one can check that $ K_n $ is autoreproducing with respect to $ m $, i.e.
\begin{align*}
\bracket{  K_n( \xb, \cdot ), K_n( \yb, \cdot )  }_{ L^2(m ) } := \int_{ \Uu^n  } K_n( \xb, z ) \overline{ K_n( \yb, z )} \ dm (z) = K_n( \xb, \yb )
\end{align*}

Writing $ dm(\xb) $ for $ dm^{\otimes k + 1}(\xb) $ to simplify the notation, one thus has
\begin{align*}
\Ee_{\Pl}^{(n)}\prth{ \abs{ \chi_\mu }^{2 k} } & = \frac{1}{n!} \int_{ (\Uu^n)^{2 k + 3} }  K_n( \xb, z ) \overline{ K_n( \yb, z )} \ dm(z) \ \overline{ f_{1^n} \!\!\otimes\!\! f_\mu^{\otimes k}(\xb)} \cdot   f_{1^n} \!\!\otimes\!\! f_\mu^{\otimes k}(\yb)  \ dm(\xb) dm(\yb) \\
			& = \frac{1}{n!} \int_{  \Uu^n  } \prth{ \int_{ (\Uu^n)^{ k + 1 } }   K_n( \xb, z ) \ \overline{ f_{1^n} \!\!\otimes\!\! f_\mu^{\otimes k}(\xb) } \ dm(\xb) } \times  \\
			& =   \hspace{+1cm} \overline{ \prth{ \int_{ (\Uu^n)^{ k + 1 } }   K_n( \yb, z ) \ \overline{ f_{1^n} \!\!\otimes\!\! f_\mu^{\otimes k}(\yb) } \ dm(\yb) } } \ dm(z) \\
			& = \frac{1}{n!} \int_{  \Uu^n  } \abs{ \int_{ (\Uu^n)^{ k + 1 } }   K_n( \xb, z ) \ \overline{ f_{1^n} \!\!\otimes\!\! f_\mu^{\otimes k}(\xb) } \ dm(\xb) }^2 dm(z)  
\end{align*}

We define the coordinatewise product of two vectors $ \xb = (x_1, \dots, x_k) $ and $ \yb = (y_1, \dots, y_k) $ by $ \xb \!\odot\! \yb := (y_1 x_1, \dots, y_k x_k) $. One can thus write 
\begin{align*}
K_n( \xb, y ) := \sum_{ \lambda \vdash n } \prth{ \bigodot \vphantom{  }_{_{\ell \geq 1}}  x_\ell}^{\lambda + \delta}  \overline{  y ^{\lambda + \delta} } = \sum_{ \lambda \vdash n } \prth{ \overline{  y } \!\odot\! \bigodot \vphantom{  }_{_{\ell \geq 1}}  x_\ell}^{\lambda + \delta} 
\end{align*}

Using the function $ R_n $, we have $ K_n(x, y) = R_n(\overline{x} \odot y) $. Then, for all $ z \in \Uu^n $,
\begin{align*}
\Le_{1^n} R_n(z) & = \int_{ \Uu^n } \overline{f_{1^n}(x) }  R_n(\overline{x} \odot  z )  \, dm(x) \\
\Le_\mu \Le_{1^n} R_n(z) & = \int_{ (\Uu^n)^2 } \overline{f_{1^n}(x_1) f_\mu(x_2) }  R_n(\overline{x_2} \odot \overline{x_1} \odot  z )  \, dm(x_1) dm(x_2) 
\end{align*}

By induction and using $ m $ in place of $ m^{ \otimes k} $, one has
\begin{align*}
\Le_\mu^k \Le_{1^n} R_n(z)   & =   \int_{ (\Uu^n)^{k + 1} } \overline{f_{1^n}(x_1) } \prod_{\ell = 2}^{k + 1} \overline{ f_\mu(x_\ell) } \,  R_n(\overline{x_1} \odot \dots \overline{x_{k + 1}} \odot  z )  \, dm(x_1, \dots, x_{k + 1}) \\
				& =  \int_{ (\Uu^n)^{k + 1} } \overline{f_{1^n}(x_1) } \prod_{\ell = 2}^{k + 1} \overline{ f_\mu(x_\ell) } \,  K_n( \, (x_1, \dots,  x_{k + 1}),  z )  \, dm(x_1, \dots, x_{k + 1}) \\
				& = \int_{ (\Uu^n)^{ k + 1 } }   K_n( \xb, z ) \ \overline{ f_{1^n} \otimes f_\mu^{\otimes k}(\xb) } \ dm(\xb)
\end{align*}

Finally, one gets 
\begin{align*}
\Ee_{\Pl}^{(n)}\prth{ \abs{ \chi_\mu }^{2 k} } = \frac{1}{n!} \int_{  \Uu^n  } \abs{  \Le_\mu^k \Le_{1^n} R_n(z) }^2 dm(z)  = \frac{1}{n!} \int_{  \Uu^n  } \abs{  \Le_\mu^k h_n(z) }^2 dm(z)
\end{align*}

Setting $ k = 0 $ in the last formula shows that $ n! = \sum_{\lambda \vdash n } d_\lambda^2 = \int_{  \Uu^n  } \abs{   h_n(z) }^2 dm(z)  $, hence
\begin{eqnarray}\label{Eq:FormuleEspCaracPlancherel}
\Ee_{\Pl}^{(n)}\prth{ \abs{ \chi_\mu }^{2 k} } = \frac{  \int_{  \Uu^n  } \abs{  \Le_\mu^k h_n(z) }^2 dm(z) }{  \int_{  \Uu^n  } \abs{  h_n(z) }^2 dm(z) }   = \frac{ \norm{ \Le_\mu^k h_n }_{L^2(\Uu^n) }^2 }{ \norm{  h_n }_{L^2(\Uu^n) }^2 }
\end{eqnarray}

By polarisation,
\begin{eqnarray}\label{Eq:FormuleEspCaracPlancherelPolarisee}
\Ee_{\Pl}^{(n)}\prth{ \prth{ \chi_\mu }^k \prth{ \overline{\chi_\nu} }^\ell } = \frac{  \int_{  \Uu^n  }   \Le_\mu^k h_n(z) \, \overline{ \Le_\nu^\ell h_n(z) } dm(z) }{  \int_{  \Uu^n  } \abs{  h_n(z) }^2 dm(z) }   = \frac{ \bracket{ \Le_\mu^k h_n, \Le_\nu^\ell h_n  }_{L^2(\Uu^n) }  }{ \norm{  h_n }_{L^2(\Uu^n) }^2 }
\end{eqnarray}
\end{proof}

\begin{remark}
The Kerov Central Limit Theorem concerns the Gaussian fluctuations of $ \widehat{\chi}_{ (k, 1^{n - k}) } := \chi_{ (k, 1^{n - k}) } / \chi_{1^n} $ when $ n \to \infty $ under $ \Pe\ell(n) $ (see \cite{IvanovOlshanskiKerov}). As we have the law of the couple $ (\chi_{ (k, 1^{n - k}) }, \chi_{1^n}) $ with \eqref{Eq:EvaluationPlancherelWithOperator}, we also have the law of $ \chi_{ (k, 1^{n - k}) } / \chi_{1^n} $. Nevertheless, the scalar product in \eqref{Eq:EvaluationPlancherelWithOperator} involves $n$ integrals, and passing to the limit in this expression is not straightforward. This problem will not occur in the case of the $ L $-measure. 
\end{remark}

\subsection{Reminders on Dirichlet characters} 

The content of this section is recalled for self-completeness and is likely to be well-known (see e.g. \cite[ch. 5]{IwanecKowalski} and references cited or \cite{SerreArithmetic, Tenenbaum}). A character of a group $G$ is a group homomorphism $ \chi : G \to \Cc^\times $, namely $ \chi(gh\inv) = \chi(g) \chi(h)\inv $ for all $ g, h \in G $. The set $ \widehat{G} $ of characters of $G$ forms a group once endowed with the pointwise multiplication $ \chi_1\cdot \chi_2(g) := \chi_1(g) \chi_2(g) $ and the inverse $ \chi\inv(g) := \chi(g)\inv $. If $ G $ is a finite group of cardinal $n$, all characters take values in the group of $n$-th roots of unity
\begin{align*}
\Uu_n := \ensemble{z \in \Cc \, / \, z^n = 1 } 
\end{align*}
as $ g^n = 1_G $ for all $ g \in G $ and $ \chi $ is a group morphism. In this case, the inverse of $ \chi $ is given by its complex conjugate, i.e. $ \chi\inv = \overline{\chi} $. In particular, a character here defined corresponds to a particular character of a dimension one representation of a group.

The main property of characters of finite abelian groups is


\begin{theorem}[Orthogonality, \cite{IwanecKowalski, Tenenbaum}] Let $G$ be a finite abelian group. Then, a basis of the space $ \Cc^G $ of $ \Cc $-valued functions on $G$ is given by the set of characters $ \widehat{G} $ that form an orthonormal basis of $ \Cc^G $ with respect to the scalar product $ \bracket{ \cdot, \cdot } $ defined by
\begin{align}\label{Eq:FirstOrthogonalityRelation}
\bracket{  f_1, f_2 } := \frac{1}{\abs{G} } \sum_{ g \in G  } f_1(g) \overline{f_2(g) }
\end{align}

One has moreover $ \abs{G} = \vert \widehat{G} \vert $. Last, setting $ \chib_g(f) := f(g) $ for the canonical evaluation on $G$, one has the dual orthogonality relation
\begin{align}\label{Eq:SecondOrthogonalityRelation}
\frac{1}{ \vert \widehat{G} \vert } \sum_{ \eta \in \widehat{G} } \chib_g(\eta) \chib_h(\eta) = \Unens{ g = h }
\end{align}
\end{theorem}


The property \eqref{Eq:SecondOrthogonalityRelation} uses \eqref{Eq:FirstOrthogonalityRelation} and the orthogonal basis expansion of the function $ \Unens{ g } : h \mapsto \Unens{ g = h } $.

$ $

We denote by $ G_q $ the group $ (\Zz/q\Zz)^\times $, and, with a slight abuse of notation, the set of Dirichlet characters by $ \widehat{G}_q $. As recalled in the introduction, this is the set of characters of $ G_q $ that are extended by periodicity to $ \Zz $. The second orthogonality relation \eqref{Eq:SecondOrthogonalityRelation} becomes thus
\begin{align}\label{Eq:ExtendedSecondOrthogonalityRelation}
\frac{1}{ \vert \widehat{G}_q \vert } \sum_{ \eta \in \widehat{G} } \chib_k(\eta) \chib_\ell(\eta) = \Unens{ k \equiv \ell \,\prth{ \Mod q } }
\end{align}
%
%
%

As a Dirichlet character can take the value $0$, we define the abstract symbol $ \partial $ such that 
\begin{align}\label{Def:AngleDirichletCharacter}
\chi(n) = e^{ i \theta_\chi(n) }, \quad\quad\quad e^{i \partial } := 0
\end{align}
that allows to define the angle $ \theta_\chi : \Zz \to \Zz \cup \ensemble{ \partial } $ of a Dirichlet character $ \chi $ .


\begin{example} 
A classical Dirichlet character is given by the Legendre symbol $ \prth{ \frac{a}{p} } $ which is a Dirichlet character modulo $p$. Define $ \operatorname{sq} : x \mapsto x^2 $ and let $ \operatorname{sq}(\Zz/p\Zz) $ be the set of squares in $ (\Zz/p\Zz)^\times $. Then, 
\begin{align*}
\prth{ \frac{a}{p} } = \Unens{a  \, \in \, \operatorname{sq}(\Zz/p\Zz) } - \Unens{ a  \, \notin \, \operatorname{sq}(\Zz/p\Zz)  } \, \in \, \ensemble{ -1, 0, 1 }
\end{align*}

Such a function is an analogue on $ (\Zz/p\Zz)^\times $ of the sign function defined by $ \operatorname{sgn} : x \in \Rr  \mapsto \Unens{ x \, \in \, \operatorname{sq}(\Rr^\times) } - \Unens{ x \, \notin \, \operatorname{sq}(\Rr^\times) } \in \ensemble{ -1, 0, 1 } $.
\end{example}


\begin{remark}
The prime factor decomposition $ q = \prod_{p \in \Pe } p^{v_p(q) } $ and the chinese remainder theorem $ (\Zz/q\Zz)^\times  \simeq \prod_{p \divise q } (\Zz/p^{v_p(q) }\Zz)^\times $ could allow to consider the sole case where $q$ is an integer power of a prime number. 
\end{remark}


Associated to a Dirichlet character $ \chi $ is the $ L $-function defined for $ s \in \ensemble{ \Re > 1 } $ by \eqref{Def:L-Function}. One can extend the definition to $ s \in \ensemble{ 1/2 \leq  \Re \leq 1 } $ by means of a functional equation (see e.g. \cite{Tenenbaum}). Equation \eqref{Def:L-Function} defines the linear form on $ \widehat{G}_q $ given by
\begin{align*}
L_s := \sum_{n \geq 1 } \frac{ 1 }{ n^s } \chib_n
\end{align*}

When $ \widehat{G}_q $ is endowed with a probability measure, these linear forms become random variables that we call for the sake of simplicity $ L $-random variables.

\subsection{Moments of characters under the $ L $-measure}

In view of \eqref{Eq:EvaluationPlancherelWithOperator}, a natural question concerns now the properties of the evaluations on $ ( \widehat{G}_q, \Pp_{s, q} ) $. The following lemma shows that the same type of structure occurs.

\begin{lemma}[Structure of evaluations and $ L $-random variables]\label{Lemma:StructureL-evaluations} For $ t \in \ensemble{ \Re > 1 } $, define the operators
\begin{align*}
D_n & : f \in L^2([0, 1]) \mapsto f(n \, \cdot )   \in L^2([0, 1]) \\
\Le_t & := \sum_{n \geq 1} \frac{1}{n^t} D_n  
\end{align*}

We recall that the polylogarithm function $ \Li_s $ is defined for $ s \in \ensemble{ \Re > 1 } $ and $ \abs{z} \leq 1 $ by
\begin{align*}
\Li_s(z) := \sum_{n \geq 1} \frac{ z^n }{n^s} 
\end{align*}

Define moreover
\begin{align*}
\mu_q & := \frac{1}{q} \sum_{k = 1}^q \delta_{ k /q } \\
\bracket{ f, g  }_{ L^2(\mu_q) } & := \frac{1}{q}  \sum_{k = 1}^q f\prth{ k/q } \overline{g\prth{ k/q }} \\
f_s  & := \Le_s e = \Li_s \circ \, e , \quad\quad e(\theta) := e^{2 i \pi \theta }
\end{align*}

Then, one has for all $ s, t, v \in \ensemble{ \Re > 1 } $ and for all $ x, y \in \Rr $ 
\begin{align}\label{Eq:EvaluationDirichletWithOperator}\begin{aligned}
\Espr{s, q}{ e^{ x \chib_k + y \overline{\chib_\ell } } } & = \frac{\bracket{ e^{ x D_k } f_s, \, e^{ y D_\ell } f_s  }_{ L^2(\mu_q) }}{ \bracket{   f_s, f_s }_{ L^2(\mu_q) } }  \\
\Espr{s, q}{ e^{ x L_t + y \overline{L_v} } } & = \frac{\bracket{ e^{ x \Le_t } f_s, \, e^{ y \Le_v } f_s  }_{ L^2(\mu_q) }}{ \bracket{   f_s, f_s }_{ L^2(\mu_q) } }  
\end{aligned}\end{align}
\end{lemma}


\begin{proof}

We only treat the case of $L$-random variables, the case of evaluations being identical. 

For $ k \geq 0 $, we have
\begin{align*}
\Ee_{s, q}\prth{ \abs{ L_t }^{2k}  }     & =    \frac{1}{ Z_{s, q} } \sum_{ \chi \in \widehat{G}_q } \abs{ L_s(\chi) }^2  \abs{ L_t(\chi) }^{2k} \\
				& =     \frac{1}{ Z_{s, q} } \sum_{ \chi \in \widehat{G}_q } \sum_{m, n \geq 1 }  \frac{ \chi(n) \overline{\chi(m) } }{ n^s m^{\overline{s} } } \sum_{ \substack{ m_1, \dots , m_k \geq 1 \\ n_1, \dots , n_k \geq 1 } } \prod_{ \ell = 1 }^k  \frac{ \chi(n_\ell) \overline{\chi(m_\ell) } }{ n_\ell^t m_\ell^{\overline{t} } } \\
				& =   \frac{1}{ Z_{s, q} } \sum_{ \substack{ m, m_1, \dots , m_k \geq 1 \\ n, n_1, \dots , n_k \geq 1 } }   \frac{ 1 }{ n^s m^{\overline{s} } } \prod_{ \ell = 1 }^k  \frac{ 1 }{ n_\ell^t m_\ell^{\overline{t} } } \sum_{ \chi \in \widehat{G}_q }  \chi\prth{ n \prod_{ \ell = 1 }^k n_\ell } \overline{\chi \prth{ m \prod_{ \ell = 1 }^k m_\ell } }  
\end{align*}				
using the multiplicativity of characters. The orthogonality relation \eqref{Eq:FirstOrthogonalityRelation} gives then			
\begin{align*}
\Ee_{s, q}\prth{ \abs{ L_t }^{2k}  }    =   \frac{ \varphi(q) }{ Z_{s, q} }  \sum_{ \substack{ m, m_1, \dots , m_k \geq 1 \\ n, n_1, \dots , n_k \geq 1 } }   \frac{ 1 }{ n^s m^{\overline{s} } } \prod_{ \ell = 1 }^k  \frac{ 1 }{ n_\ell^t m_\ell^{\overline{t} } } \Unens{  n n_1 \cdots n_k \equiv m m_1 \cdots m_k \, \prth{ \Mod q } }  
\end{align*}

Recall the well known equality 
\begin{align*}
\Unens{ a \in q \Zz} = \frac{1}{q} \sum_{k = 1}^q e^{2 i \pi a k / q } = \int_0^1 e^{2 i \pi a \theta } \mu_q(d\theta), \qquad \mu_q := \frac{1}{q} \sum_{k = 1}^q \delta_{k/q}
\end{align*}

It implies
\begin{align*}
\Ee_{s, q}\prth{ \abs{ L_t }^{2k}  } & =   \frac{ \varphi(q) }{ Z_{s, q} }   \sum_{ \substack{ m, m_1, \dots , m_k \geq 1 \\ n, n_1, \dots , n_k \geq 1 } }   \frac{ 1 }{ n^s m^{\overline{s} } } \prod_{ \ell = 1 }^k  \frac{ 1 }{ n_\ell^t m_\ell^{\overline{t} } }  \int_0^1 e^{ 2 i \pi \theta ( n \, n_1  \cdots   n_k -  m \, m_1  \cdots   m_k )  } \mu_q(d\theta) \\
				& =   \frac{ \varphi(q) }{ Z_{s, q} }   \int_0^1 \sum_{  n, n_1, \dots , n_k \geq 1   }   \frac{ e^{ 2 i \pi \theta  n \, n_1  \cdots   n_k } }{ n^s \, (n_1  \cdots   n_k)^t  }   \sum_{  m, m_1, \dots , m_k \geq 1   }   \frac{ e^{ - 2 i \pi \theta  m \, m_1  \cdots   m_k } }{ m^{ \overline{s} } \, (m_1  \cdots   m_k)^{ \overline{t} }  } \mu_q(d\theta) \\
				& =   \frac{ \varphi(q) }{ Z_{s, q} }   \int_0^1 \abs{ \sum_{  n, n_1, \dots , n_k \geq 1   }   \frac{ e^{ 2 i \pi \theta  n \, n_1  \cdots   n_k } }{ n^s \, (n_1  \cdots   n_k)^t  }   }^2  \mu_q(d\theta)
\end{align*}

Here, we have used the Fubini theorem relative to sums and integrals which is justified by the absolute convergence of the sums for $ s, t \in \ensemble{ \Re > 1 } $.

Using $ e : \theta \mapsto e^{2 i \pi \theta } $ and $ D_n e(\theta) := e(n\theta) $, one can thus write
\begin{align*}
\Ee_{s, q}\prth{ \abs{ L_t }^{2k}  } = \frac{ \varphi(q) }{ Z_{s, q} } \norm{  \sum_{  n, n_1, \dots , n_k \geq 1   }   \frac{ D_{  n \, n_1  \cdots   n_k } e }{ n^s \, (n_1  \cdots   n_k)^t  }  }^2_{ L^2(\mu_q) }
\end{align*}

Since $ D_{mn} = D_m D_n = D_n D_m $, one can write
\begin{align*}
\Ee_{s, q}\prth{ \abs{ L_t }^{2k}  } = \frac{ \varphi(q) }{ Z_{s, q} } \norm{   \prth{  \sum_{  n  \geq 1   } \frac{1}{ n^t } D_n }^k \sum_{  n  \geq 1   } \frac{1}{ n^s } D_n  e  }^2_{ L^2(\mu_q) } = \frac{ \varphi(q) }{ Z_{s, q} } \norm{ \Le_t^k \Le_s e  }^2_{ L^2(\mu_q) }
\end{align*}

Setting $ k = 0 $, one finds that $  \norm{  \Le_s e  }^2_{ L^2(\mu_q) } =  Z_{s, q} / \varphi(q)  $ and using $ f_{s, q} := \Le_s e / \norm{  \Le_s e  }_{ L^2(\mu_q) }$, one thus has
\begin{align}\label{Eq:MomentsL-functionsNonPolarised}
\Espr{s, q}{ \abs{ L_t }^{2k}  }   =  \frac{ \norm{ \Le_t^k \Le_s e  }^2_{ L^2(\mu_q) }  }{ \norm{ \Le_s e  }^2_{ L^2(\mu_q) } } = \norm{ \Le_t^k f_{s, q}  }^2_{ L^2(\mu_q) }
\end{align}

Moreover, one has 
\begin{align*}
\frac{ Z_{s, q} }{ \varphi(q) } \Ee_{s, q}\prth{ \abs{ L_t }^{2k}  }  & =  \sum_{ \substack{ m, m_1, \dots , m_k \geq 1 \\ n, n_1, \dots , n_k \geq 1 } }   \frac{ 1 }{ n^s m^{\overline{s} } } \prod_{ \ell = 1 }^k  \frac{ 1 }{ n_\ell^t m_\ell^{\overline{t} } } \Unens{  n n_1 \cdots n_k \equiv m m_1 \cdots m_k \, \prth{ \Mod q } }  \\
                 & \leq \sum_{ \substack{ m, m_1, \dots , m_k \geq 1 \\ n, n_1, \dots , n_k \geq 1 } }   \frac{ 1 }{ n^s m^{\overline{s} } } \prod_{ \ell = 1 }^k  \frac{ 1 }{ n_\ell^t m_\ell^{\overline{t} } } = \abs{ \zeta(s) \zeta(t)^k }^2
\end{align*}
so $ \sum_{k \geq 0} \Ee_{s, q}\prth{ \abs{ L_t }^{2k} } \rho^k / k! < \infty $ for all $ \rho \in \Rr_+ $. By polarisation and generating series in $ x, y $, one thus finds the result.
\end{proof}


\begin{remark} The auto-reproducing kernel equivalent to \eqref{Def:AutoreproducingKernelPermutations} is here
\begin{align*}
K(n_1, \dots, n_k \, ; \, m_1, \dots, m_k) := \sum_{ \chi \in \widehat{G}_q } \prod_{ \ell = 1 }^k \chi(n_\ell) \overline{\prod_{ \ell = 1 }^k \chi(m_\ell)} = \varphi(q) \Unens{ n_1 \cdots n_k \equiv m_1 \cdots m_k \, \prth{ \Mod q }  }
\end{align*}

It is autoreproduced for $ L^2(\mu_q) $ in the sense that
\begin{align*}
K(n_1, \dots, n_k \, ; \, m_1, \dots, m_k) = \bracket{  K(n_1, \dots, n_k \, ; \,q \cdot ), K(m_1, \dots, m_k \, ; \, q \cdot )  }_{ L^2(\mu_q) }
\end{align*}
since we defined $ \mu_q = q\inv \sum_{k = 1}^q \delta_{k/q} $ in place of $ \mu_q = q\inv \sum_{k = 1}^q \delta_k $ (hence the multiplication by $q$). The autoreproduction is a consequence of the fact that $ \abs{ \chi(a) }^2 = 1  $ for all $ a \in \intcrochet{1, q} $.
\end{remark}

\section{Fluctuations when $ q \to +\infty $}

\subsection{Evaluations}\label{SubSec:FluctuationsEvaluations}  

We now prove theorem \ref{Theorem:ConvergenceAngles}.

\begin{proof}
We use the method of moments. We have 
\begin{align*}
\Espr{s, q}{ \chib_m^k \overline{\chib}_m^\ell }  =  \frac{ \bracket{  D_m^k f_s ,  D_m^\ell  f_s   }_{ L^2(\mu_q) }  }{ \bracket{ f_s , f_s   }_{ L^2(\mu_q) } }, \quad\quad\quad f_s :=  \Li_s \circ\, e 
\end{align*}

Each term of the ratio has the same form. It is thus enough to prove the convergence of the numerator when $ q \to +\infty $. As 
\begin{align*}
\bracket{  D_m^k f_s ,  D_m^\ell  f_s   }_{ L^2(\mu_q) } = \int f_s\prth{e^{ 2 i \pi m^k \theta } } \overline{f_s\prth{e^{ 2 i \pi m^\ell \theta } } } \mu_q(d\theta) =: \frac{1}{q}  \sum_{j = 1}^q g_{m, s}^{(k, \ell) } \prth{\frac{j}{q } }
\end{align*}
we have a Riemann sum, thus, it is enough to show that the function $ \theta \mapsto g_{m, s}^{(k, \ell) }(\theta) $ is continuous and integrable on $ \prth{0, 1} $ to get the convergence. 
This fact is clear, given the form of the function $ f_s = \Li_s\!\circ e $ (bounded by $ \abs{\zeta (s) } $ on $ \crochet{0, 1} $), hence the result :
\begin{align}\label{Eq:ConvergenceRiemannSum}
\int_0^1 g_{m, s}^{(k, \ell) }(\theta) d\mu_q(\theta) \tendvers{q}{+\infty }   \int_0^1 g_{m, s}^{(k, \ell) }(\theta) d\theta
\end{align}

We now identify the limit. Set $ \mu_\infty(d\theta) = \Unens{ 0 \leq \theta \leq 1} d\theta $. Then,
\begin{align*}
\bracket{  D_m^k f_s ,  D_m^\ell  f_s   }_{ L^2(\mu_\infty) } & = \int_0^1 f_s(e^{ 2 i \pi m^k \theta } ) \overline{f_s\prth{e^{ 2 i \pi m^\ell \theta } } }  d\theta  \\
              & = \sum_{n_1, n_2 \geq 1 } \frac{1}{ (n_1 n_2)^s } \Unens{ n_1 m^k  = n_2 m^\ell } \\
              & = m^{ks } m^{\ell s } \sum_{n_1, n_2 \geq 1 } \frac{1}{ (n_1 m^k  n_2 m^\ell )^s } \Unens{ n_1 m^k = n_2 m^\ell } \\ 
              & = m^{ks } m^{\ell s } \sum_{n_1  \geq 1 } \frac{1}{ (n_1 m^k )^{2s} }  \sum_{n_2 \geq 1 } \Unens{ n_1 m^k  = n_2 m^\ell } \\ 
              & = m^{-ks } m^{\ell s } \sum_{n \geq 1} \frac{1}{ n^{2 s} } \Unens{ m^\ell  \divise m^k  n }
\end{align*}

Note that the case $ k = \ell = 0 $ gives $ \bracket{   f_s ,    f_s   }_{ L^2(\mu_\infty) }  =  \zeta(2s)  $.
\begin{align*}
\mbox{Now, } \qquad \bracket{ D_m^k f_s , D_m^\ell f_s }_{ L^2(\mu_\infty) } & = m^{ ( \ell - k) s } \sum_{n \geq 1 } \frac{1}{ n^{2 s} } \Unens{ m^\ell \divise m^k  n } \\
               & = m^{ ( \ell - k ) s } \sum_{n \geq 1 } \frac{1}{ n^{2 s} } \prth{ \Unens{ \ell \leq k } + \Unens{ \ell > k, \, m^{\ell - k } \divise n } } \\
               & = \Unens{ \ell \leq k } \frac{1}{ m^{ (k - \ell ) s } } \zeta(2s) + \Unens{\ell > k} m^{ (\ell - k) s } \sum_{n \geq 1 } \frac{1}{  n^{2 s} } \Unens{  m^{\ell - k } \divise n }   \\ 
               & = \Unens{ \ell \leq k } \frac{1}{ m^{ (k - \ell ) s } } \zeta(2s) + \Unens{\ell > k} m^{ (\ell - k) s }\sum_{d \geq 1 } \frac{1}{  (d  m^{\ell - k } )^{2 s} } \\
               & = \Unens{ \ell \leq k } \frac{1}{ m^{ (k - \ell ) s } } \zeta(2s) + \Unens{\ell > k} \frac{1}{ m^{ (\ell - k) s } }  \zeta(2s) \\
               & = \frac{1}{ m^{ \abs{ \ell - k } s } } \zeta(2s)
\end{align*}

Thus, writing $ \Espr{s, \infty }{  \cdot } $ for the limiting measure, we get
\begin{align*}
\Espr{s, \infty}{ \chib_m^k \overline{\chib}_m^\ell }  =  \frac{\bracket{  D_m^k f_s ,  D_m^\ell  f_s   }_{ L^2(\mu_\infty) }}{ \bracket{ f_s , f_s   }_{ L^2(\mu_\infty) } } =  e^{- s \log ( m )   \abs{ k - \ell }  } 
\end{align*}
%
%
%

The symmetric stable random variable $ St(\alpha) $ of parameter $ \alpha \in (0, 2] $ is defined for all $ x \in \Rr $ by
\begin{align*}
\Esp{ e^{i x \, St(\alpha)}  } = e^{ - \abs{x}^\alpha } 
\end{align*}

In the particular case of $ \alpha = 1 $,  $ St(1) = \Ce $, a Cauchy-distributed random variable whose Lebesgue-density is given by $ \Prob{ \Ce \in dx } = \frac{1}{1 + x^2} \frac{dx}{\pi } $. Thus, 
\begin{align*}
\Espr{s, \infty}{ \chib_m^k \overline{\chib}_m^\ell   }   =   \Esp{ e^{ i ( k - \ell )  s \log ( m ) \Ce    } } 
\end{align*}
hence the result.
\end{proof}


\begin{remark}\label{Rk:ConvergenceRationalFunctions}
It is enough to prove the convergence \eqref{Eq:ConvergenceRiemannSum} on rational functions of the disk of the type $ z \mapsto \prod_{k = 1}^m \frac{1}{z^j - x_k} $ for $ 0 < x_k < 1 $. Indeed, one has the following integral representation of the polylogarithm
\begin{align*}
\Li_s(z) & = \sum_{n \geq 1} \frac{z^n}{n^s} = \sum_{n \geq 1} z^n \frac{1 }{\Gamma(s) } \int_{\Rr_+ } e^{-nt } t^s \frac{dt}{t} = \frac{1 }{\Gamma(s) } \int_{\Rr_+ }  \sum_{n \geq 1} z^n e^{-nt } t^s \frac{dt}{t} \\
              & = \frac{1 }{\Gamma(s) } \int_{\Rr_+ } \frac{z e^{-t } }{1 - z e^{-t} } t^s \frac{dt}{t} 
\end{align*}
as $ \abs{z } = 1 $ and $ e^{-t } < 1 $ for $ t > 0 $. Using the Fubini theorem, it is thus enough to prove 
\begin{align*}
\int_\Uu  \frac{ z^{k_1}   }{1 - z^{k_1} e^{-t_1} }  \frac{z^{-k_2}   }{1 - z^{-k_2} e^{-t_2} }  \mu_q \prth{ \frac{d^*z}{z } }  \tendvers{q}{ + \infty } \int_\Uu  \frac{ z^{k_1}   }{1 - z^{k_1} e^{-t_1} }  \frac{ z^{-k_2} }{1 - z^{-k_2} e^{-t_2} }  \frac{d^*z}{z }
\end{align*}
which is clear. As we have moreover

\begin{align*}
\int_\Uu  \frac{ z^{k_1} }{1 - z^{k_1} e^{-t_1} }  \frac{ z^{-k_2} }{1 - z^{-k_2} e^{-t_2} }  \frac{d^*z}{z } & = \sum_{r, \ell \geq 0 } e^{ - (r t_1 + \ell  t_2) } \int_\Uu z^{k_1 (r + 1) - k_2(\ell + 1) }  \frac{d^*z}{z } \\
              & = e^{t_1 + t_2}  \sum_{r, \ell \geq 1 } e^{ - (r t_1 + \ell  t_2) } \Unens{ k_1 r = k_2 \ell }
\end{align*}
we end up with the following equality that can be proven with a direct computation
\begin{align*}
\int_{ \Rr_+^2 }  \sum_{\ell_1,  \ell_2 \geq 1 } e^{ - (\ell_1 t_1 + \ell_2  t_2) } \Unens{ k^m \ell_1 = k^r \ell_2 } \, t_1^s \frac{dt_1}{t_1}   \, t_2^s \frac{dt_2}{t_2}  = \Gamma(s) k^{ s \abs{ m - r } }
\end{align*}
\end{remark}

\subsection{Joint evaluations} 

%
%


We now prove theorem \ref{Theorem:JointEvaluationFluctuations}.

\begin{proof}
Set $ \kappa := \prod_{j = 1}^\ell p_j^{k_j} $ and $ \rho := \prod_{j = 1}^\ell p_j^{r_j} $ for $ (k_j)_j $ and $ (r_j)_j $ integers. Then, 
\begin{align*}
\Espr{q, s}{ \prod_{j = 1}^\ell \chib_{p_j}^{k_j}   \overline{ \chib_{p_j}^{r_j} } } & = \Espr{q, s}{   \chib_\kappa \overline{ \chib_\rho } } \tendvers{q}{+\infty } \frac{1}{\zeta(s) } \sum_{n_1, n_2 \geq 1} \frac{1}{(n_1 n_2)^s } \Unens{\kappa n_1 = \rho n_2 }
\end{align*}

Now, recalling that $ a_+ := \max(a, 0) $, define
\begin{align*}
N_1 := \prod_{j = 1}^\ell p_j^{ (k_j - r_j)_+ },  \qquad  N_2 := \prod_{j = 1}^\ell p_j^{ (r_j - k_j)_+ }
\end{align*}
so that $ \ensemble{ \kappa n_1 = \rho n_2 } = \ensemble{ N_1 n_1 = N_2 n_2 } $ and $ \gcd(N_1, N_2) = 1 $. We have 
\begin{align*}
\sum_{n_1, n_2 \geq 1 } \frac{1}{(n_1 n_2)^s } \Unens{ N_1 n_1 = N_2 n_2 } & = (N_1 N_2)^s \sum_{n_1  \geq 1 } \frac{1}{ (n_1 N_1)^{2s} } \Unens{ N_2 \divise n_1 N_1 } \\
              & = (N_1 N_2)^s \sum_{n_1  \geq 1 } \frac{1}{ (n_1 N_1)^{2s} } \Unens{ N_2 \divise n_1   } \quad \mbox{as }\gcd(N_1, N_2) = 1 \\
              & = (N_1 N_2)^s \sum_{ n  \geq 1 } \frac{1}{ (N_2 n N_1)^{2s} } = \frac{\zeta(s) }{(N_1 N_2)^s } \\
              & = \zeta(s) \prod_{j = 1}^\ell p_j^{ -s [ (k_j - r_j)_+ + (r_j - k_j)_+ ] }  \\
              & = \zeta(s) \prod_{j = 1}^\ell p_j^{ -s \abs{ k_j - r_j} }
\end{align*}
hence the result.
\end{proof}

\subsection{Speed of convergence in the Wasserstein metric}

One can adapt the proof of theorem \ref{Theorem:ConvergenceAngles} to get the speed of convergence in the Wasserstein distance. 

\begin{lemma}[Speed of convergence in the Wasserstein metric]\label{Lemma:SpeedOfConvergence} For all bounded continuous functions $ h : \Cc \to \Rr $, we have
\begin{align}\label{Eq:WassersteinDistance}
\Espr{s, q}{ h(  \chib_m )  } = \Esp{ h\prth{ e^{ i s \log(m) \Ce } }  } + O\prth{  \frac{\zeta(s)^2 }{\zeta(2s) } \frac{  \norm{ h  }_\infty }{ q^s } }
\end{align}
\end{lemma}


\begin{proof}
We have 
\begin{align*}
\bracket{  D_m^k f_s ,  D_m^\ell  f_s   }_{ L^2(\mu_q) } & = \sum_{n_1, n_2 \geq 1 } \frac{1}{(n_1 n_2)^s } \Unens{ n_1 m^k \equiv n_2 m^\ell (\Mod q) }  \\
              & = \sum_{n_1, n_2 \geq 1 } \frac{1}{(n_1 n_2)^s } \sum_{q \in \Zz} \Unens{ n_1 m^k = n_2 m^\ell  + r q} \\
              & = \sum_{n_1, n_2 \geq 1 } \frac{1}{(n_1 n_2)^s } \Unens{ n_1 m^k = n_2 m^\ell   }  \\
              &  \hspace{+2cm} + m^{ s(k + \ell ) } \sum_{q \in \Zz^*} \sum_{n_1, n_2 \geq 1 } \frac{1}{(n_1 m^k n_2 m^\ell)^s } \Unens{ n_1 m^k = n_2 m^\ell  + r q} \\
              & = \frac{ \zeta(2s) }{ m^{s \abs{k - \ell} } } +  m^{ s(k + \ell ) } \sum_{ r, n \geq 1  } \prth{ \frac{1}{(n m^k (n m^k + rq) )^s } + \frac{1}{(n m^\ell (n m^\ell + rq) )^s }  }
\end{align*}

Thus, 
\begin{align*}
0 \leq \bracket{  D_m^k f_s ,  D_m^\ell  f_s   }_{ L^2(\mu_q) } - \frac{ \zeta(2s) }{ m^{s \abs{k - \ell} } } & \leq   m^{ s(k + \ell ) } \sum_{ r, n \geq 1  } \prth{ \frac{1}{(n m^k rq  )^s } + \frac{1}{(n m^\ell  rq  )^s }  } \\
               & = \frac{ m^{sk} + m^{s\ell } }{ q^s } \zeta(s)^2 
\end{align*}

And the case $ k = \ell = 0 $ gives
\begin{align*}
0 \leq \bracket{  f_s , f_s   }_{ L^2(\mu_q) } - \zeta(2s)   & = 2 \sum_{ r, n \geq 1  }  \frac{1}{(n (n + rq ) )^s }   \leq \frac{ 2 }{ q^s } \zeta(s)^2
\end{align*}

This implies 
\begin{align*}
\Espr{s, q}{ \chib_m^k \overline{\chib_m}^\ell }  & = \frac{\bracket{  D_m^k f_s ,  D_m^\ell  f_s   }_{ L^2(\mu_q) }}{\bracket{ f_s , f_s   }_{ L^2(\mu_q) } }  \\
                & = \frac{   \frac{ \zeta(2s) }{ m^{s \abs{k - \ell} } } + O\prth{ \frac{ m^{sk} + m^{s\ell } }{ q^s } \zeta(s)^2  } }{  \zeta(2s) + O\prth{ \frac{ 1 }{ q^s } \zeta(s)^2 }   } \\
                & = \Esp{ e^{i (k - \ell) s \log(m) \Ce} } + O\prth{  \frac{ m^{sk} + m^{s\ell } }{ q^s } \frac{\zeta(s)^2 }{ \zeta(2s) } }
\end{align*}

Taking linear combinations on $ (k, \ell) $, we deduce that for all real polynomial $  P, Q $, we get
\begin{align*}
\Espr{s, q}{ P(\chib_m ) Q(\overline{\chib_m}) }  = \Esp{ P\prth{ e^{i  s \log(m) \Ce} } Q\prth{ e^{-i  s \log(m) \Ce} }  } + O\prth{  \frac{ P(m^s) Q(1) + Q(m^s)P(1) }{ q^s } \frac{\zeta(s)^2 }{ \zeta(2s) } }
\end{align*}

The implicit constant in the $ O $ does not depend on any parameter. We can thus take a sequence of polynomials that approximate bounded continuous functions $ f, g $ and get 
\begin{align*}
\Espr{s, q}{ f(\chib_m ) g(\overline{\chib_m}) }  = \Esp{ f\prth{ e^{i  s \log(m) \Ce} } g\prth{ e^{-i  s \log(m) \Ce} }  } + O\prth{  \frac{ \norm{f}_\infty \norm{g}_\infty }{ q^s } \frac{\zeta(s)^2 }{ \zeta(2s) } }
\end{align*}

We thus get the result.
\end{proof}

\begin{remark}
We also have the integral representation for $ k, \ell \geq 0 $
\begin{align*}
\bracket{  D_m^k f_s ,  D_m^\ell  f_s   }_{ L^2(\mu_q) } & = \frac{ \zeta(2s) }{ m^{s \abs{k - \ell} } } +  \sum_{ r, n \geq 1  } \prth{ \frac{ m^{ s \ell } }{ n^s   (n m^k + rq )^s } + \frac{ m^{ s k } }{ n^s   (n m^\ell + rq )^s }  } \\
              & = \frac{ \zeta(2s) }{ m^{s \abs{k - \ell} } } +   \sum_{ r, n \geq 1  } \int_{ \Rr_+ }\prth{ \frac{ m^{ s \ell } }{ n^s}  e^{-t (n m^k + rq ) } + \frac{ m^{ s k } }{ n^s}  e^{-t (n m^\ell + rq ) } } t^{s - 1} \frac{dt}{\Gamma(s) } \\
              & = \frac{ \zeta(2s) }{ m^{s \abs{k - \ell} } } +   \frac{1}{q^s}  \int_{ \Rr_+ } \frac{ m^{ s \ell }   \Li_s(e^{-t m^k/q } ) +  m^{ s k } \Li_s( e^{-t m^\ell / q } ) }{1 - e^{- t} }  \, e^{- t }   t^{s - 1} \frac{dt}{\Gamma(s) }
\end{align*}

Using $ \overline{\chib_m} = \chib_m\inv $, this implies, for all $ k \in \Zz $ 
\begin{align*}
\Espr{s, q}{ \chib_m^k } & = \frac{\zeta(2s) }{ \zeta_q(2s) } \Esp{ e^{i k s\log(m)\, \Ce } } \\
            & \qquad +\frac{1}{q^s \zeta_q(s) } \int_{ \Rr_+ } \frac{ m^{ - s \abs{k}  }   \Li_s(e^{-t m^{ \abs{k}  }/q } ) +  m^{ s \abs{k}  } \Li_s( e^{-t m^{-\abs{k} } / q } ) }{1 - e^{- t} }  \, e^{- t }   t^{s - 1} \frac{dt}{\Gamma(s) }  
\end{align*}
where $ \zeta_q(s) :=  \bracket{  f_s , f_s   }_{ L^2(\mu_q) } = \zeta(2s) + 2 \sum_{ r, n \geq 1  }  \frac{1}{(n (n + rq ) )^s }$.

This last representation, nevertheless, does not allow to extend analytically $ k \mapsto \Espr{s, q}{ \chib_m^k } $ as the function $ k  \mapsto \abs{k} $ is not analytic in $ 0 $. 
\end{remark}

\begin{question}\label{Question:RangeOfs}
The results on the speed of convergence of the windings of the complex Brownian motion as exposed in \cite{PapYor, BentkusPapYor} show a speed of convergence in $ O(t^{-1/2}) $ in the total variation distance. If we make the analogy between $t$ and $q$, we should thus consider the case $ s = 1/2 $. It is possible to give a meaning to the values of $ L $-functions in such a number, by means of a functional equation (see e.g. \cite{Tenenbaum}). This raises the question of the behaviour of random characters for $s$ in this range of values. 
\end{question}

\subsection{The case $ s = \infty $ : the uniform measure}\label{SubSec:OtherMeasure} 

\subsubsection{Evaluations}

Using the classical result
\begin{align*}
\abs{ L_s(\chi) }^2 = \prod_{p \in \Pe} \abs{ 1 + \frac{\chi(p) }{p^s }  }^{-2} \tendvers{s}{+\infty } 1
\end{align*}
we can conclude that 
\begin{align*}
\Proba{\!s, q}{\chi} = \frac{ \abs{L_s(\chi) }^2 }{ \sum_{ \eta \in \widehat{G}_q } \abs{L_s(\eta) }^2  } \tendvers{s}{+\infty } \frac{1}{ \varphi(q) } = \Proba{\infty, q}{\chi}
\end{align*}

This is the uniform measure on $ \widehat{G}_q $ that assigns equal mass to any character. The behaviour of the evaluations $ \chib_k $ under this measure is given in theorem \ref{Theorem:LimitEvaluationsUniformMeasure} that we recall

\vspace{+0.25cm}

\noindent\textbf{Theorem \ref{Theorem:LimitEvaluationsUniformMeasure}} \textit{The following convergence in law is satisfied for all fixed integer $ k \in \intcrochet{2, q} $} 
\begin{align*}
\chib_k \cvlaw{q}{ + \infty } e^{i 2\pi  U}, \qquad U \sim \Us\!\prth{\, \crochet{0, 1}}
\end{align*}

\textit{Moreover, $ (\chib_p)_{p \in \Pe} $ converges in law to a vector of independent random variables.}
%


\begin{proof}

An application of the method of moments and the second orthogonality relation gives 
\begin{align*}
\Esp{  \chib_k^m \overline{ \chib}_k^\ell  } & =  \frac{1}{\varphi(q) }  \sum_{ \chi \in \widehat{G}_q }  \chi(k)^m \overline{ \chi(k)}^\ell \\
             & = \frac{1}{\varphi(q) }  \sum_{ \chi \in \widehat{G}_q }  \chi(k^m) \overline{ \chi(k^\ell) } \\
             & = \Unens{ k^m \equiv k^\ell \,\prth{ \Mod q } } \\
             & \tendvers{q}{+\infty } \Unens{ k^m = k^\ell } = \Unens{  m =  \ell } = \int_0^1 e^{2 i \pi (m - \ell) \theta } d\theta = \Esp{ e^{2 i \pi (m - \ell) U } }
\end{align*}

For the independence, set $ n := \prod_{p \in \Pe} p^{m_p} $ and $ n' := \prod_{p \in \Pe} p^{\ell_p} $ for $ (m_p)_p $ and $ (\ell_p)_p $ two sequences of integers. Then,
\begin{align*}
\Esp{ \prod_{p \in \Pe } \chib_p^{m_p} \overline{ \chib}_p^{\ell_p} \!\! } =  \frac{1}{\varphi(q) } \! \sum_{ \chi \in \widehat{G}_q } \! \chi(n)  \overline{ \chi(n')} & = \Unens{ n \equiv n' \,\prth{ \Mod q } } \\
             & \tendvers{q}{+\infty } \Unens{ n = n' } \!= \!\prod_{p \in \Pe } \! \Unens{  m_p =  \ell_p } \!=\! \prod_{p \in \Pe } \!  \Esp{ e^{2 i \pi (m_p - \ell_p) U } }
\end{align*}
which concludes the proof.
\end{proof}

We remark that this last convergence does not depend on $k$, as long as $k$ does not vary with $q$. This thus asks the

\begin{question}\label{Question:PhaseTransition}
Describe the phase transition between $ s < \infty $ and $ s = \infty $, i.e. can we find $ s \equiv s(\lambda, q) \to +\infty $ when $ q \to +\infty $ with a certain parameter $ \lambda \in \Rr $ such that $ \chib_k $ converges in law to a non-trivial random variable $ e^{i \Theta(\lambda, k) } $ ?

The law of such a random variable $ e^{i\Theta(\lambda, k)} $ would interpolate between  $ e^{i s \log(k) \Ce } $ and $ e^{ 2 i \pi U } $. 
\end{question}

\subsubsection{The Bohr-Jessen distribution}

An interested question related with the uniform measure $  \Pp_{\! \infty, q} $ is to compute the limiting distribution of $ L_t(\chib) $, as a classical distribution will appear at the limit :

\begin{theorem}[Convergence in law of $ L_t(\chib) $ under $ \Pp_{\! \infty, q} $] Define the Bohr-Jessen law for $ \alpha > 1 $ by
\begin{align}\label{Def:BohrJessenLaw}
\BJ_\alpha \eqlaw \prod_{p \in \Pe } \prth{ 1 - \frac{e^{2 i \pi U_p } }{p^\alpha } }\inv 
\end{align}
where $ (U_p)_{p \in \Pe} $ is a sequence of i.i.d. uniform random variables on $ \crochet{0, 1} $. Then, one has under $ \Pp_{\! \infty, q} $
\begin{align*}
L_t(\chib ) \cvlaw{q}{+\infty } \BJ_t
\end{align*}
\end{theorem}


\begin{proof}
We have for all $ t > 1 $
\begin{align*}
\Espr{\infty, q}{ \abs{ L_t(\chib) }^{2k} } & = \sum_{ \substack{  m_1, \dots , m_k \geq 1 \\ n_1, \dots , n_k \geq 1 } } \prod_{i = 1}^k  \frac{1}{ (n_i m_i)^t }  \Unens{ \prod_{i = 1}^k n_i \equiv  \prod_{i = 1}^k m_i \, ( \Mod q ) } \\
               & \tendvers{q}{+\infty } \sum_{ \substack{  m_1, \dots , m_k \geq 1 \\ n_1, \dots , n_k \geq 1 } } \prod_{i = 1}^k  \frac{1}{ (n_i m_i)^t }  \Unens{ \prod_{i = 1}^k n_i =  \prod_{i = 1}^k m_i } = \sum_{n \geq 1} \prth{ \frac{ \Un^{*k}(n) }{n^t } }^2 
\end{align*}
hence the convergence in law. To identify the limit, we use 
\begin{align*}
\Un^{*k}(n) = \prod_{p \in \Pe } f_k( v_p(n) ), \qquad f_k(m) := \frac{k^{ \uparrow m } }{ m! }
\end{align*}
which is a consequence of the identities, valid for all $ z \in \ensemble{\Re > 1} $
\begin{align*}
\sum_{n \geq 1} \frac{\Un^{*k}(n)}{ n^z } \stackrel{\mbox{\eqref{Eq:ConvolutionMorphism}} }{=} \! \prth{\sum_{n \geq 1} \frac{1 }{ n^z } }^{\!\!\!  k} \! \stackrel{\mbox{\eqref{Eq:EulerProd=SumGeneral}} }{=} \prod_{p \in \Pe } \prth{ 1 -  p^{ -z } }^{-k } = \prod_{p \in \Pe } \sum_{ m_p \geq 0 } \frac{k^{ \uparrow m_p } }{ m_p! }  p^{ -m_p z  } \stackrel{\mbox{\eqref{Eq:ConvolutionMorphism}} }{=} \sum_{n \geq 1} \frac{1}{n^z } \prod_{p \in \Pe } \frac{k^{ \uparrow v_p(n) } }{ v_p(n)! } 
\end{align*}
\vspace{-0.2cm}
\begin{align*}
\mbox{We can thus write }   \sum_{n \geq 1} \prth{ \frac{ \Un^{*k}(n) }{n^t } }^2   & =  \sum_{n \geq 1} \frac{ 1 }{n^{2t} } \prth{ \prod_{p \in \Pe } f_k( v_p(n) ) }^2 \\
               & = \prod_{p \in \Pe } \sum_{ m_p \geq 0 }  f_k( m_p )^2  p^{ -2 m_p t  } \quad \mbox{ using \eqref{Eq:ConvolutionMorphism}}  \\
               & = \prod_{p \in \Pe } \int_0^1 \abs{ \sum_{ m \geq 0 }  f_k( m )   p^{ -  m t  } e^{2 i \pi m u }  }^2 \!\! du  \mbox{ (Parseval formula)}\\
               & = \prod_{p \in \Pe } \Esp{ \abs{ 1 - p^{-t } e^{ 2 i \pi U } }^{-2k} } = \Esp{ \abs{\BJ_t}^{2k} }
\end{align*}
which concludes the proof.
\end{proof}

\begin{remark}
The Bohr-Jessen distribution \eqref{Def:BohrJessenLaw} appeared in \cite{BohrJessen} when considering the limiting distribution of the random variable $ \zeta( \alpha + i T U) = L_{\alpha + i TU}(1) $ for $ \alpha > 1/2 $ and $ T \to +\infty $. When $ \alpha > 1 $, this is the method of moments that is used. The behaviour inside the strip $  1/2 < \alpha < 1 $ is studied by means of an approximate functional equation (see \cite[II.]{BohrJessen}). 

\end{remark}

\begin{remark}
To compute the speed of convergence of the moments, one writes
\begin{align*}
\Espr{\infty, q}{ \abs{ L_t(\chib) }^{2k} } & = \sum_{ m, n \geq 1}  \frac{ \Un^{*k}(n) \Un^{*k}(m) }{(mn)^t }  \Unens{ m \equiv  n \, ( \Mod q ) }	\\
             & = \Esp{ \abs{\BJ_t}^{2k} } + 2 \sum_{  n, \ell \geq 1}  \frac{ \Un^{*k}(n + q \ell) \Un^{*k}(n) }{(n + q \ell)^t n^t }  
\end{align*}

To pursue, one needs an equivalent of $ \Un^{*k}(n) $ when $ n \to +\infty $. For $k = 2$, we have $\Un^{*2}(n) = d(n) := \sum_{d \divise n} 1 $, the number of divisors of $n$. It is known that $ d(n) = o_\varepsilon(n^\varepsilon) $ for all $ \varepsilon > 0 $, hence the speed of convergence 
\begin{align*}
\Espr{\infty, q}{ \abs{ L_t(\chib) }^2 } = \Esp{ \abs{\BJ_t}^2 } + O_{\varepsilon, t}\prth{ \frac{1}{ q^{t - \varepsilon} } }  
\end{align*}
\end{remark}

\section{Another type of $L$-measure}\label{Section:LMesureGenerale}

As seen in the introduction, the general analogue of the Frobenius characteristic 
\begin{align*}
\Ch_X(\chi ) = \frac{1}{n!} \sum_{\sigma \in \Sg_n } \chi (\sigma) \, \prod_{k \geq 1} p_\ell(X)^{ m_\ell(\sigma) }  
\end{align*}
is
\begin{align*}
L_\ab(\chi ) := \sum_{ n \in \Nn^* } \chi (n) \, \prod_{p \in \Pe } a_p^{ v_p(n) }  
\end{align*}
where $ (a_p)_{p \in \Pe} $ is a sequence of real numbers satisfying $ 0 < a_p < 1 $ and $ \prod_{p \in \Pe} \abs{1 - a_p}\inv < \infty $. We thus define the the measure 
\begin{align}\label{Def:GeneralLmeasure}
\Proba{\ab ; q }{\chi} := \frac{ \abs{ L_\ab(\chi ) }^2 }{ Z(\ab) }, \qquad Z(\ab) := \sum_{ \chi \in \widehat{G}_q } \abs{ L_\ab(\chi ) }^2
\end{align}

These last functions may seem more general, but they do not enjoy, in their full generality, the same properties as the $L$-function, for instance a functional equation that allows to extend them analytically ; hence, in a certain way, the $ L $-measures defined in \ref{Def:DirichletMeasure} are more general. For a certain choice of parameters $ (a_p)_p $, though, one can get $L$-functions of \textit{twisted characters} (see e.g. \cite{SerreArithmetic, Tenenbaum}).

A more general definition with for instance $ L_{\bb}(\chi) := \sum_n \chi(n) b_n $ with $ \sum_n \abs{b_n} < \infty $ is also possible, but the limiting evaluations will not be independent ; in the case of \eqref{Def:GeneralLmeasure}, we indeed have the

\begin{theorem} Let $ \ell \geq 1 $ and $ (p_j)_{1 \leq j \leq \ell} $ be fixed prime numbers. Then, under $ \Pp_{\ab ; q} $, we have the following convergence in distribution
\begin{align*}
( \chib_{ p_1}, \dots, \chib_{ p_\ell } ) \cvlaw{q}{ +\infty } \prth{ e^{i    \log(a_{p_1}) \, \Ce_1} , \dots, e^{i    \log(a_{p_\ell}) \, \Ce_\ell } }
\end{align*}
where $ (\Ce_1, \dots, \Ce_\ell) $ are independent Cauchy-distributed random variables. 
\end{theorem}


\begin{proof}
Set $ \kappa := \prod_{j = 1}^\ell p_j^{k_j} $ and $ \rho := \prod_{j = 1}^\ell p_j^{r_j} $ for $ (k_j)_j $ and $ (r_j)_j $ integers. Then, using the second orthogonality relation, the multiplicativity of characters and the property $ v_p(n_1 n_2) = v_p(n_1) + v_p(n_2) $ we get

\begin{align*}
\Espr{\ab ; s}{ \prod_{j = 1}^\ell \chib_{p_j}^{k_j}   \overline{ \chib_{p_j}^{r_j} } } & = \Espr{\ab ; s}{   \chib_\kappa \overline{ \chib_\rho } } \tendvers{q}{+\infty } \frac{1}{ Z(\ab) } \sum_{n_1, n_2 \geq 1} \prod_{p \in \Pe} a_p^{v_p(n_1 n_2) } \Unens{\kappa n_1 = \rho n_2 }
\end{align*}

Define $ N_1 := \prod_{j = 1}^\ell p_j^{ (k_j - r_j)_+ } $ and $  N_2 := \prod_{j = 1}^\ell p_j^{ (r_j - k_j)_+ } $ 
so that $ \ensemble{ \kappa n_1 = \rho n_2 } = \ensemble{ N_1 n_1 = N_2 n_2 } $ and $ \gcd(N_1, N_2) = 1 $. We have 
\begin{align*}
\sum_{n_1, n_2 \geq 1 } \prod_{p \in \Pe} a_p^{v_p(n_1 n_2) } \Unens{ N_1 n_1 = N_2 n_2 } & = \prod_{p \in \Pe} a_p^{- v_p(N_1 N_2) } \sum_{n_1  \geq 1 } \prod_{p \in \Pe} a_p^{2 v_p(N_1 n_1) } \Unens{ N_2 \divise n_1 N_1 } \\
              & = \prod_{p \in \Pe} a_p^{- v_p(N_1 N_2) } \sum_{n_1  \geq 1 } \prod_{p \in \Pe} a_p^{2 v_p(N_1 n_1) } \Unens{ N_2 \divise n_1   } \quad \mbox{as }\gcd(N_1, N_2) = 1 \\
              & = \prod_{p \in \Pe} a_p^{- v_p(N_1 N_2) } \sum_{n \geq 1 } \prod_{p \in \Pe} a_p^{2 v_p(N_1 n N_2) } = Z(\ab) \prod_{p \in \Pe} a_p^{ v_p(N_1 N_2) } \\
              & = Z(\ab) \prod_{j = 1}^\ell a_{p_j}^{   (k_j - r_j)_+ + (r_j - k_j)_+ }  \\
              & = Z(\ab) \prod_{j = 1}^\ell a_{p_j}^{ \abs{ k_j - r_j} }
\end{align*}
hence the result.
\end{proof}

%

\section{Conclusion and perspectives}

The probabilistic fluctuations that we have addressed in this article are amongst the most natural questions when one is concerned with a particular model of random variables. In the case of random Dirichlet characters, though, several open questions deserve to be tackled, for instance the phase transition when $s$ depends on $q$ (question \ref{Question:PhaseTransition}) or the extension of the range of $s$ (question \ref{Question:RangeOfs}), not to mention the extension to other types of characters such as characters of function fields, Hecke characters or higher-dimensional representations. One can also restrict the framework to primitive characters only.

Other probabilistic questions of interest concern the large deviations of $ \theta_\chib(p) $ or its mod-Cauchy convergence (see e.g. \cite{DelbaenAl}) a rescaling of a diverging Fourier transform. In the case of the windings of the complex Brownian motion $ Z_t = R_t e^{i \Theta_t } $, one has, locally uniformly in $ u \in \Rr $ (see \cite{BarhoumiModRandomisation})
\begin{align}\label{Eq:ModCauchyConvergenceWinding}
\frac{ \Esp{ e^{i u \Theta_t } } }{  \Esp{ e^{i u \log(\sqrt{t}) \Ce  } } }  \tendvers{t}{ + \infty }  \Phi_\Theta(u) := 2^{ -\abs{u} / 2} \frac{\Gamma(1 + \abs{u}/2  ) }{\Gamma(1 + \abs{u}) }
\end{align}

This result implies the convergence in law towards the Cauchy distribution, and is related with a second-order type of convergence in distribution (see \cite{BarhoumiModStein}). In the case of $ \theta_\chib(p) $, one needs to ``rescale'' by a diverging factor $ \sigma(q) \in \Nn $ and look for a convergence of the form 
\begin{align*}
\frac{ \Espr{q, s}{ \chib_p^{ k \sigma(q) } } }{ \Esp{ e^{i k \sigma(q) \Ce} } } \tendvers{q}{ + \infty } \Phi_{s, \chib_p}(k)
\end{align*}

A last probabilistic question concerns a Markovian coupling in $q$ of the considered random variables. On the symmetric group $ \Sg_n $, such a coupling exists for the Schur measure and is related with the notion of induction of a representation (or Pieri rule, see e.g. \cite{Kerov}). In the case of $ G_q $, one can consider $ q = p^\nu $ for $ p \in \Pe $ and $ \nu $ that varies in $ \Nn $ to look for such an inductive coupling. 

We will tackle these questions in subsequent publications.



\section*{Acknowledgements}

The author expresses its thanks to the following persons for their useful discussions, corrections and encouragements~: Ga\"etan Borot, Valentin F\'eray, Olivier H\'enard, Jeffrey Kuan, Mark Masdeu, Joseph Najnudel, Simon P\'epin-Lehalleur, Nina Snaith and Nikos Zygouras. 

The author was supported by EPSRC grant EP/L012154/1.

\bibliographystyle{amsplain}


\end{document}